\documentclass[a4paper,12pt]{amsart}
\usepackage{marginnote}

	\usepackage{amsmath,amsthm,amsfonts,amssymb,mathrsfs}
   \usepackage{graphicx}
   \usepackage{color}
	\usepackage{delarray}
	\usepackage{enumerate}
	\usepackage{hyperref}

   \newtheorem{theo}{Theorem}[section]
	\newtheorem{prop}[theo]{Proposition}
   \newtheorem{lemm}[theo]{Lemma}

   	\newtheorem{assumption}[theo]{Assumptions}
	\theoremstyle{remark}

   \def \u{\bar{u}}
	\def \n{\bar{n}}

   \def \div{ \nabla \cdot}

	\def\norm#1#2{\|#1\|_{L^#2(\Omega)}}
   \newcommand{\dx}{\,{\rm d}x}
   \newcommand{\dy}[1]{\,{\rm d}#1}
	\newcommand\R{\mathbb{R}}

\numberwithin{equation}{section}
\numberwithin{theo}{section}

\begin{document}

\title[Model of chemotactic E.~coli colonies]{ Mathematical treatment of  PDE model \\
 of  chemotactic {\it E.~coli} colonies}

\author[R. Celi\'nski]{Rafa\l\ Celi\'nski}
\address[R. Celi\'nski]{
 Instytut Matematyczny, Uniwersytet Wroc\l awski,
 pl. Grunwaldzki 2/4, 50-384 Wroc\-\l aw, POLAND}
\email{Rafal.Celinski@math.uni.wroc.pl}

\author[D. Hilhorst]{Danielle Hilhorst}
\address[D. Hilhorst]{
CNRS et Laboratoire de Math\' ematiques, Universit\' e de Paris-Sud,
B\^at.~307, 91405 Orsay Cedex, FRANCE}
\email{Danielle.Hilhorst@math.u-psud.fr}

\author[G. Karch]{Grzegorz Karch}
\address[G. Karch]{
 Instytut Matematyczny, Uniwersytet Wroc\l awski,
 pl. Grunwaldzki 2/4, 50-384 Wroc\-\l aw, POLAND}
\email{Grzegorz.Karch@math.uni.wroc.pl}
\urladdr{http://www.math.uni.wroc.pl/~karch}

\author[M. Mimura]{Masayasu Mimura}
\address[M. Mimura]{
 Graduate School of Integrated Sciences for Life, Hiroshima University, 1-3-1 Kagamiyama, Higashi-Hiroshima City,
Hiroshima 739-8526, JAPAN}
\email{mimura.masayasu@gmail.com}
\urladdr{http://home.mims.meiji.ac.jp/~mimura}

\author[P. Roux]{Pierre Roux}
\address[P. Roux]{
	Laboratoire de Math\' ematiques, Universit\' e de Paris-Sud,
	B\^at.~307, 91405 Orsay Cedex, FRANCE}
\email{pierre.roux@math.u-psud.fr}

\date{\today}

\begin{abstract}{
We consider an initial-boundary value problem for 
reaction-diffusion equations coupled with the Keller-Segel system from the chemotaxis theory which 
describe a formation of
colony patterns of bacteria {\it Escherichia coli}. 
The main goal of this work is to show that  global-in-time solutions of this model converge towards stationary solutions depending on initial conditions.}
%
%
%
%
\end{abstract}

\bigskip


\maketitle

\section{Introduction}

Budrene and Berg \cite{BB91,BB91_2} performed experiments showing  that chemotactic strains of bacterias {\it E.~coli}, inoculated in semi-solid agar, form stable and remarkably complex but geometrically regulated spatial patterns such as swarm rings, radial spots, interdigitated arrays of spots and rather complex chevron-like patterns as shown in Fig.~1. They suggested that such colonial patterns depend on an initial concentration of a nutrient (substrate) which determines how long  multicellular aggregate structures remain active. They expected that four elements such as: 
the substrate consumption, the cell proliferation, the excretion of attractant, and the chemotactic motility, when they are suitably combined, can generate complex spatial structures in a self-organized way and that a specialized and more complex morphogenetic program is not required. However, this hypothesis does not necessarily imply that such complex patterns occur as a consequence of self-organization.
\begin{figure}
\includegraphics[scale=0.7]{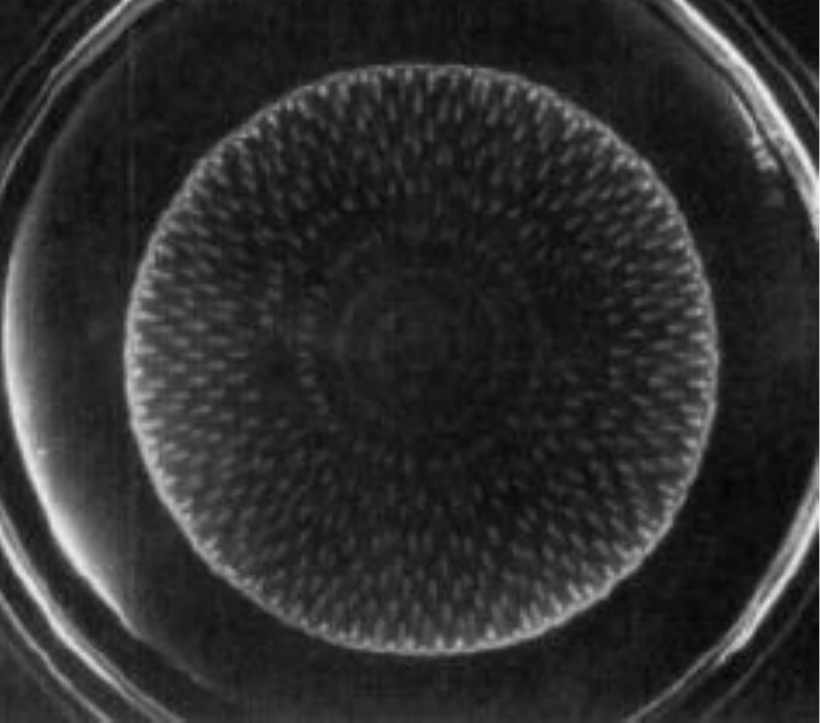}
\caption{Chevron-like pattern in experiment \cite{BB91, BB91_2}}
(by courtesy of Budrene and Berg)
\label{fig:exp}
\end{figure}

It is a challenging problem in the field of mathematical biology to understand self-organization and, in particular, the influence of a chemotaxis on the occurrence of colonial patterns. 
The  mathematical model of chemotaxis was introduced by Keller
and Segel \cite{KS70,KS71} which simplified version consists of the following system of partial differential equations 
\begin{equation}\label{KS0}
\begin{split}
&u_t= d_u\Delta u-\div(u\nabla c )\\
&c_t=d_c\Delta c+\alpha u-\beta c,
\end{split}
\end{equation}
where $u=u(x,t)$ denotes the density of cells and $c=c(x,t)$ is a concentration of chemoattractant.
Since the seminal papers  by Keller and Segel, a great number of chemotaxis PDE models have been  introduced and studied several papers.  
Here, we quote only the monograph \cite{yagi}, the  reviews \cite{H,BBTW}
as well as the papers \cite{Cao, Corrias2, HW, JL92, N01, TW, W2, W3, W4,W} with mathematical results related to those in this work.

In order to model a pattern formation in bacteria colonies, 
 Mimura and Tsujikawa \cite{MT96} considered  more general model based on the chemotaxis and growth of bacteria:
\begin{equation}\label{i1}
\begin{split}
&u_t= d_u\Delta u-\div(u\nabla\chi(c))+f(u)\\
&c_t=d_c\Delta c+\alpha u-\beta c.
\end{split}
\end{equation}
Here, 
 $\chi=\chi(c)$ is the sensitivity  function of chemotaxis and $f(u)$ is a growth function with an Allee effect. 
The authors of \cite{MT96} studied the influence of the form of $\chi(c)$ on the occurrence of chemotaxis-induced instability but  model \eqref{i1} {could not generate patterns} similar to those  observed by Budrene and Berg.
Other mathematical results on model \eqref{i1} can be found {\it e.g.~}in \cite{TW, W3} and in references therein.

Another approach to model   Budrene and Berg experiments consists in the handling the concentration of a nutrient $n(t,x)$ which leads to the system of three equations
\begin{equation}\label{intr1}
\begin{split}
&u_t= d_u\Delta u-\div(u\nabla\chi(c))+g(u)nu,\\
&c_t=d_c\Delta c+\alpha u-\beta c,\\
&n_t=d_n\Delta n-\gamma g(u)nu.
\end{split}
\end{equation}
In fact, this approach appears in other  models which are basically similar to the one in  \eqref{intr1}, see {\it e.g.}  \cite{T95,LMT99,SK00,P06}.
The authors of these works  suggest that the chemotactic effect generates
spotty patterns which are a consequence of a chemotaxis-induced instability.
 However, they have not shown that such models
generate geometrically regulated patterns which are observed in experiments when an initial nutrient is changed.

For this reason, Mimura and his collaborators  \cite{AMM10} proposed a new system of differential equations  with two internal states of bacteria: active and less-active ones. 
Denoting  the density of active bacteria  by $u(x,t)$, the density of inactive bacteria by $w(x,t)$, the density of nutrient by $n(x,t)$, and the concentration of chemoattractant by $c(x,t)$,  the new diffusion-chemotaxis-growth system has the form
\begin{equation}\label{ecoli1i}
\begin{split}
&u_t= \Delta u-\div(u\nabla\chi(c))+g(u)nu-b(n)u, \\
&c_t=d_c\Delta c+\alpha u-\beta c, \\
&n_t=d_n\Delta n-\gamma g(u)nu, \\
&w_t= b(n)u.
\end{split}
\end{equation}
In the beginning of the next section, we formulate assumptions which are imposed on   parameters and functions in equations \eqref{ecoli1i}.
Here, we only  remark that the first three equations are closed for $u, c$ and $n$, and the density  $w$ can be obtained from $u$ and $n$ by the formula
\begin{equation}\label{w:0}
w(x,t) = \int_0^t b(n(x,\tau))u(x,\tau) d \tau,
\end{equation}
where the initial condition $w(x,0) = 0$ is required from  experiments. Then, the resulting colonial pattern is represented by the total density  $u(x,t)+w(x,t)$.
{The main goal of this work is to show that 
  there exists an asymptotic
 inactive bacteria configuration $w_\infty\in L^\infty(\Omega)$ such that
\begin{equation*}
u(x,t)\; \xrightarrow{t\to\infty}\; 0 \qquad \text{and}\qquad   w (x,t) \; \xrightarrow{t\to\infty}\;  w_\infty (x)
\end{equation*}
which can be regarded as a formation  of a stationary colonial pattern.}

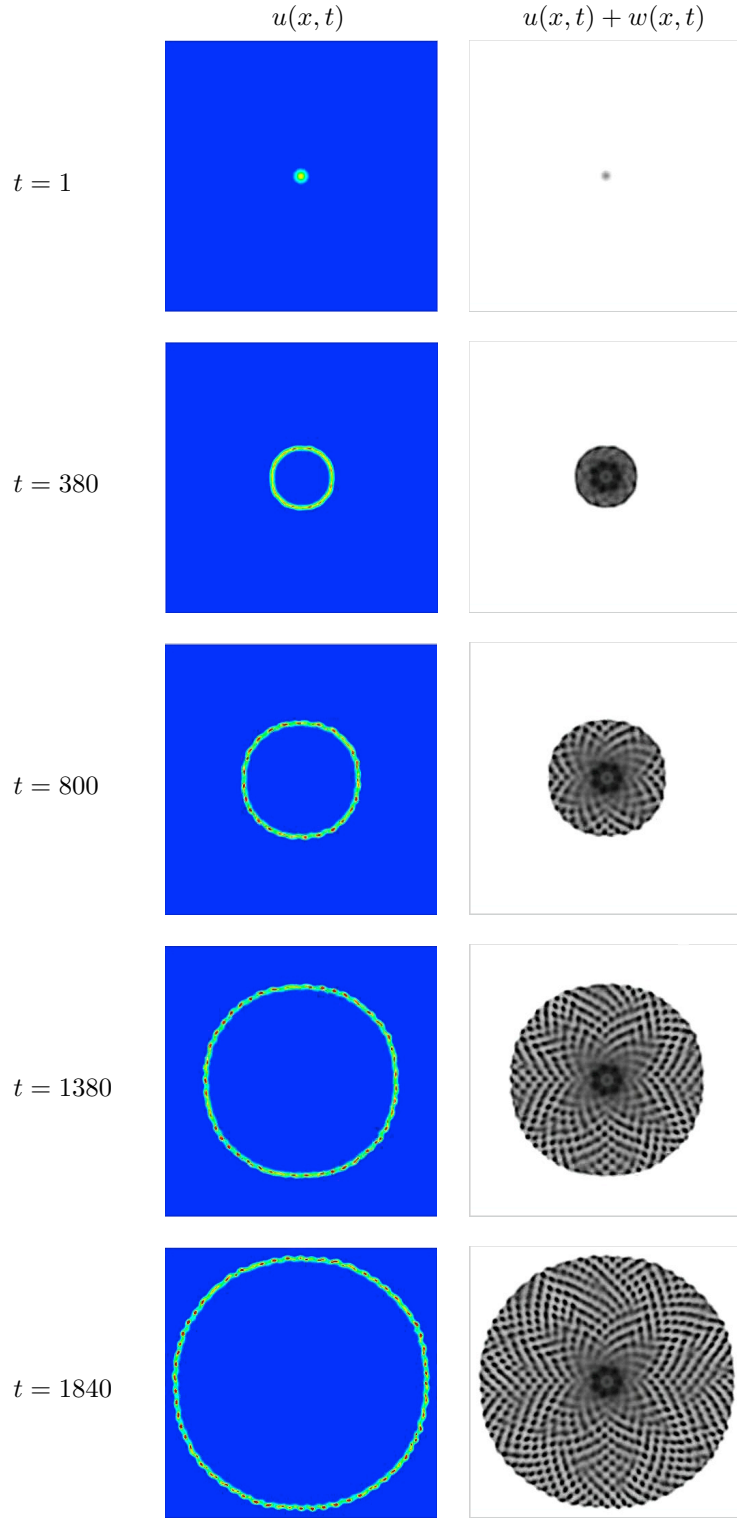
\begin{figure}
 \setlength{\unitlength}{1mm}
{\footnotesize
\begin{picture}(60,200)
 \put(-5,160){\includegraphics[scale=0.3]{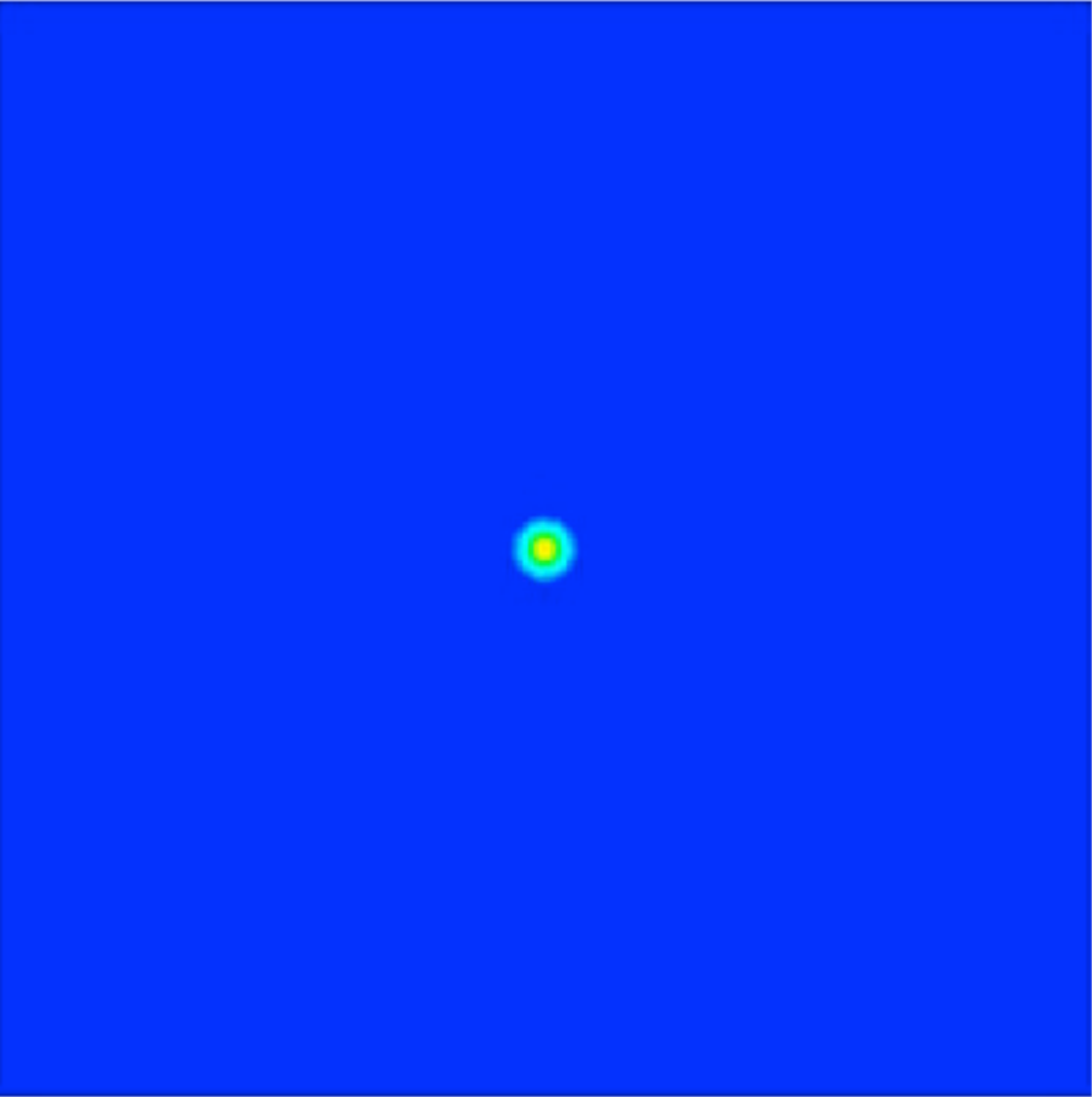}}
 \put(35,160){\includegraphics[scale=0.3]{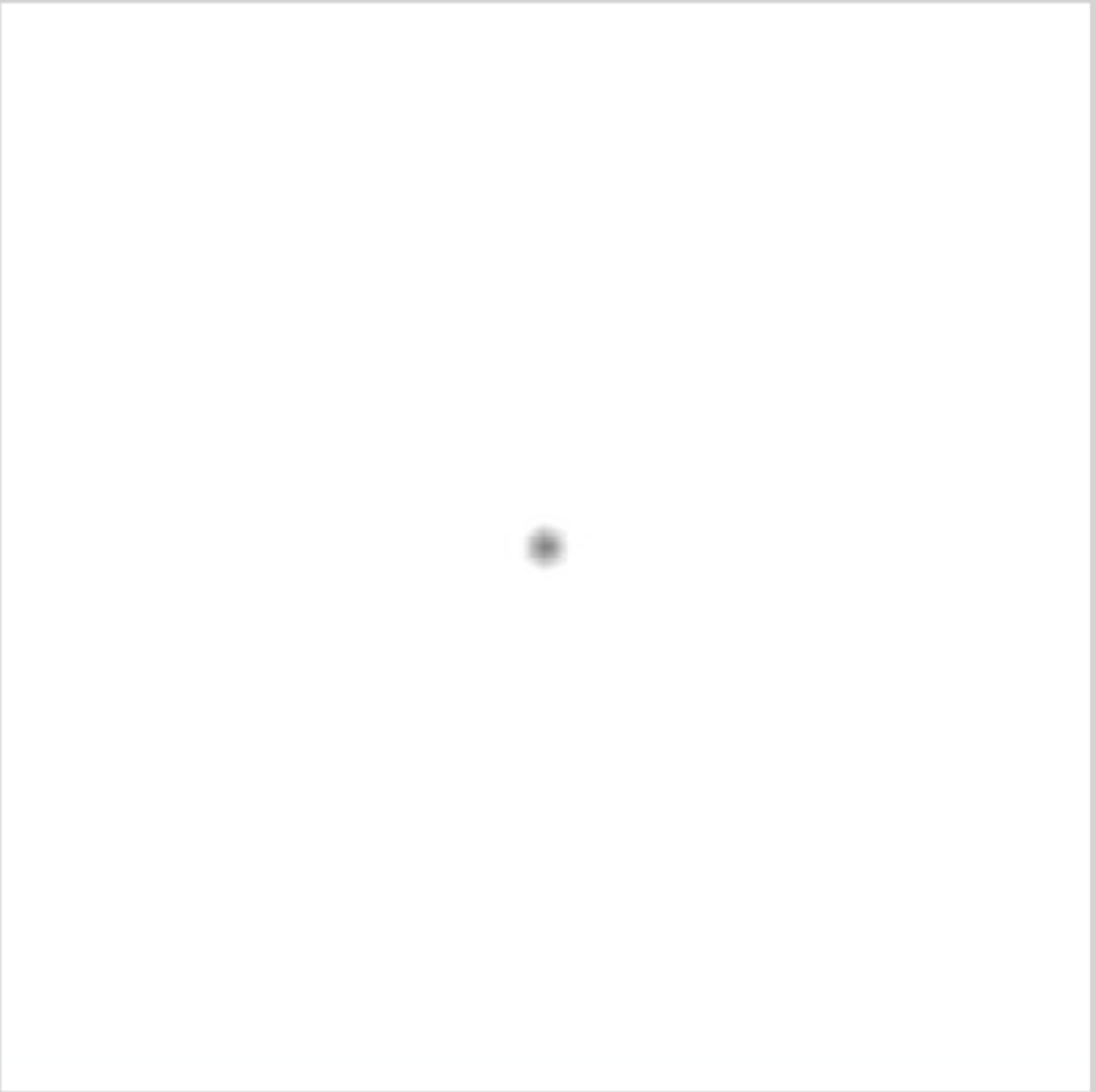}}

 \put(-5,120){\includegraphics[scale=0.3]{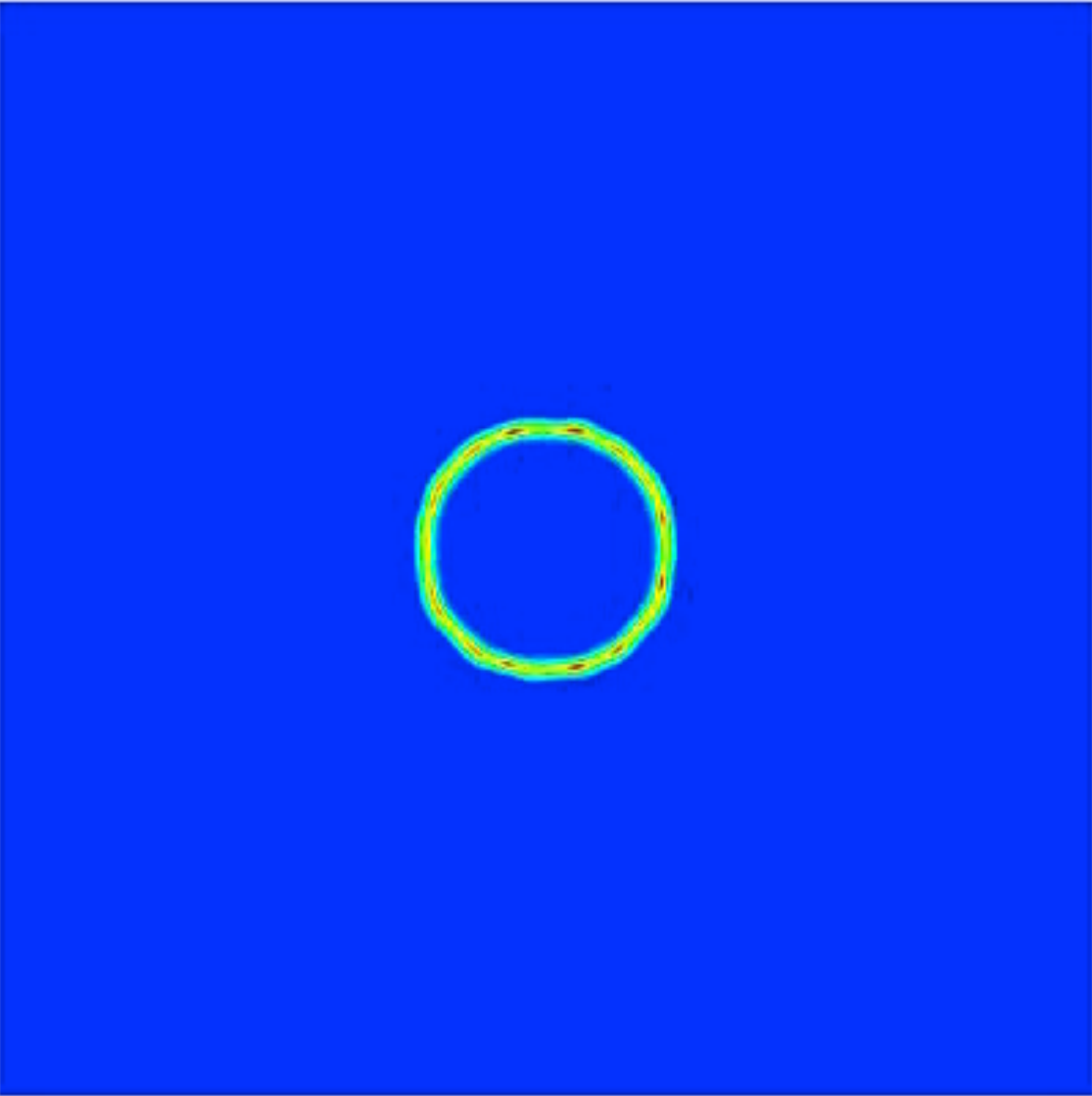}}
 \put(35,120){\includegraphics[scale=0.3]{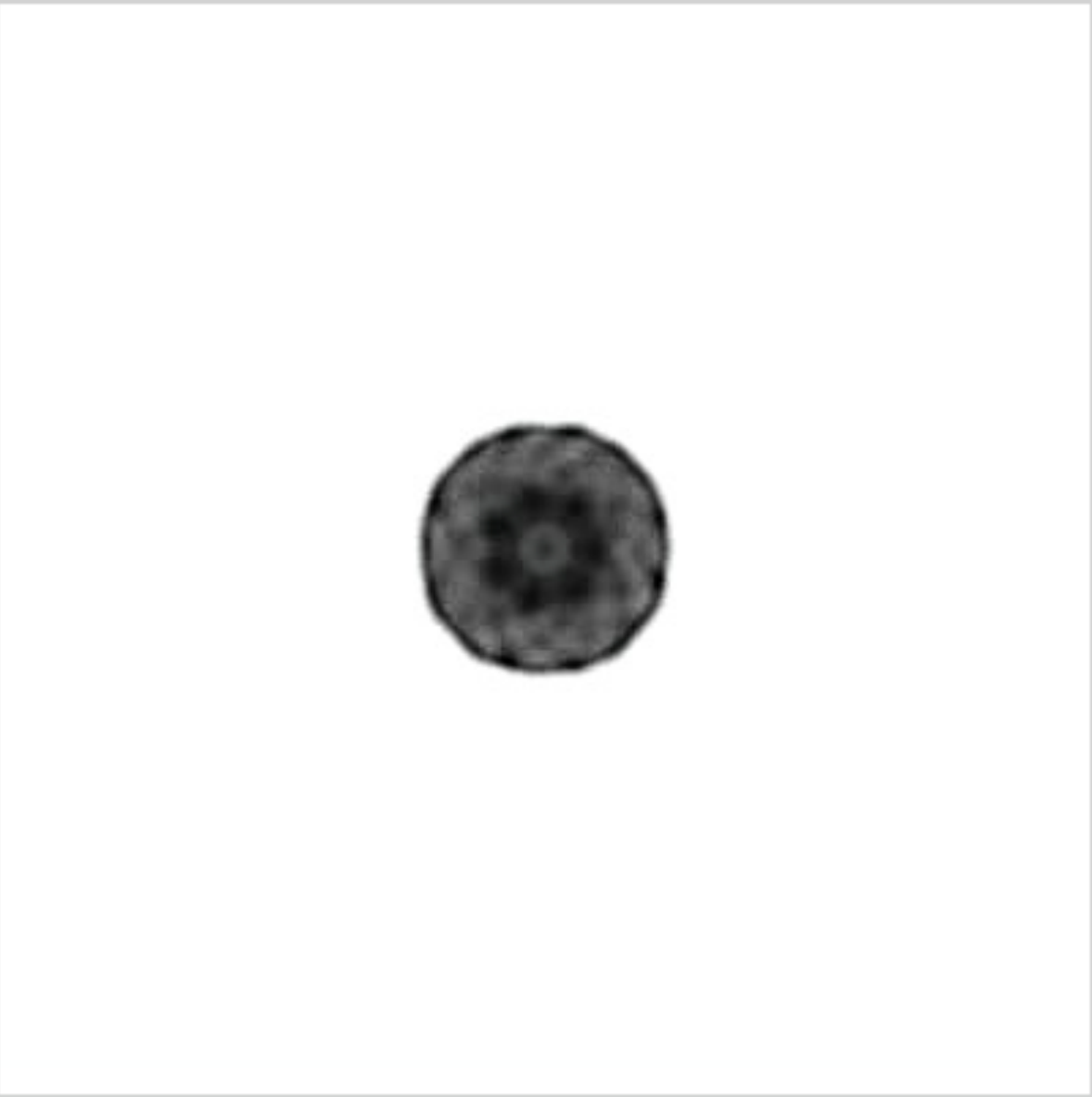}}

 \put(-5,80){\includegraphics[scale=0.3]{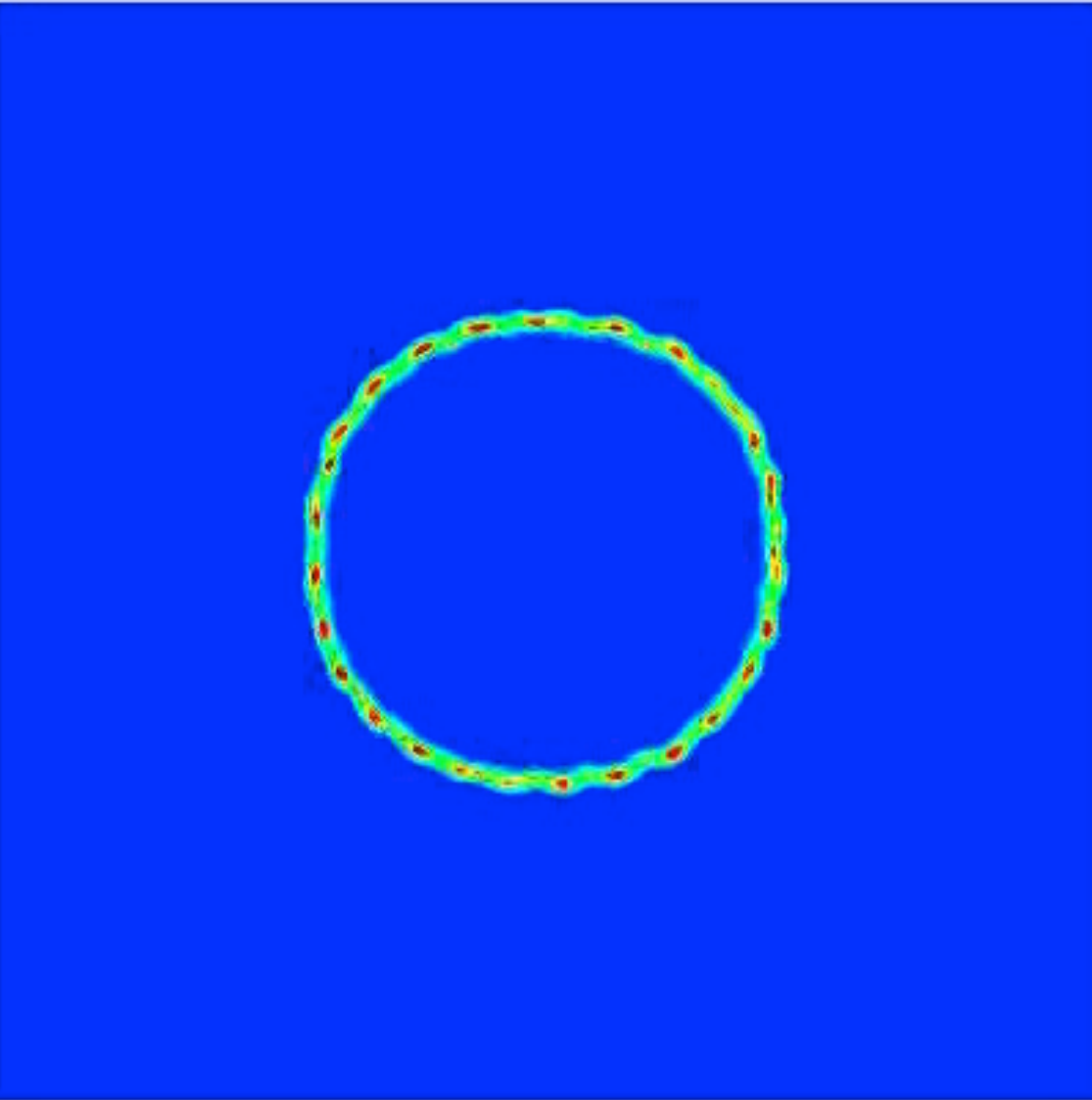}}
 \put(35,80){\includegraphics[scale=0.3]{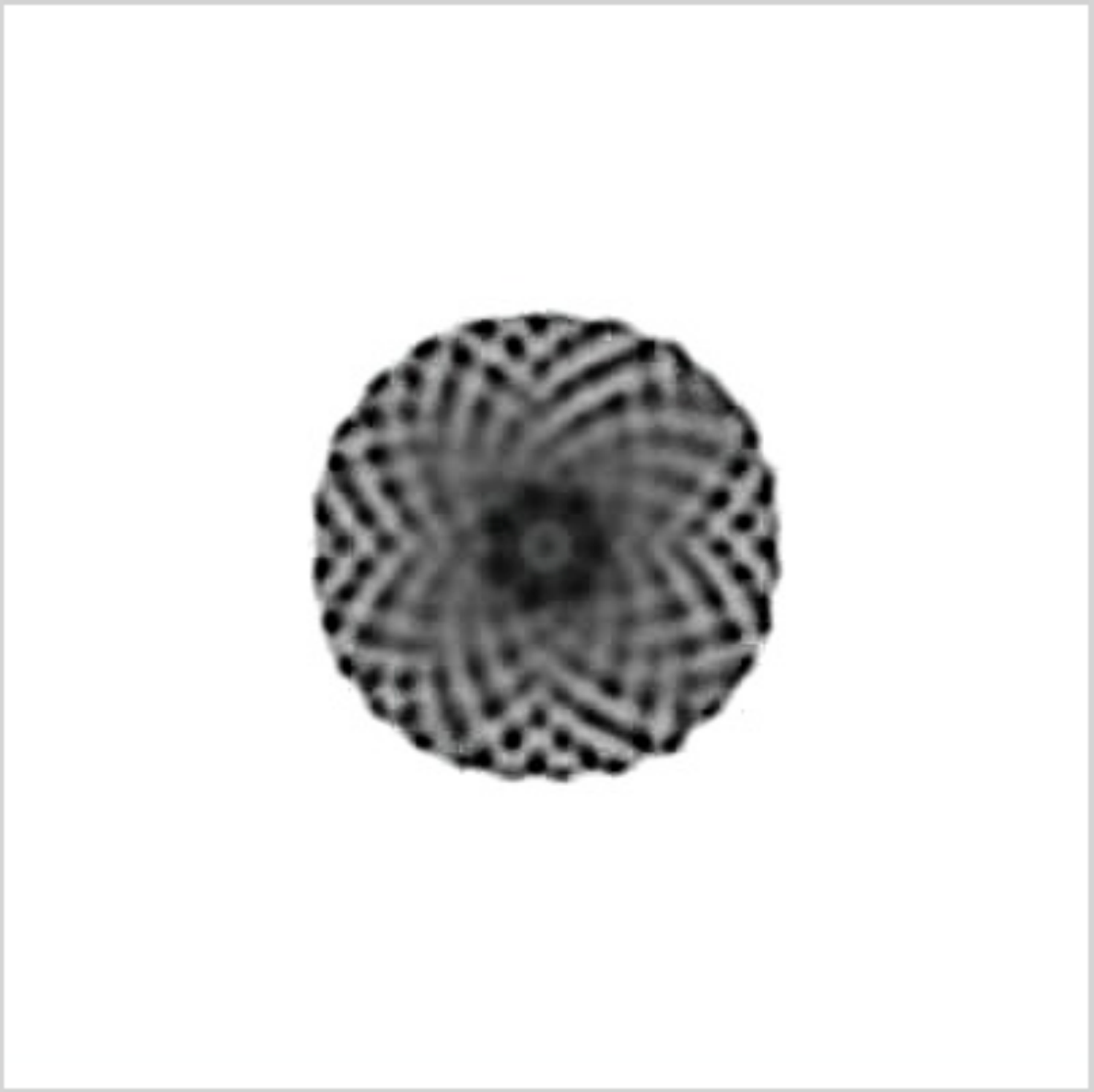}}

 \put(-5,40){\includegraphics[scale=0.3]{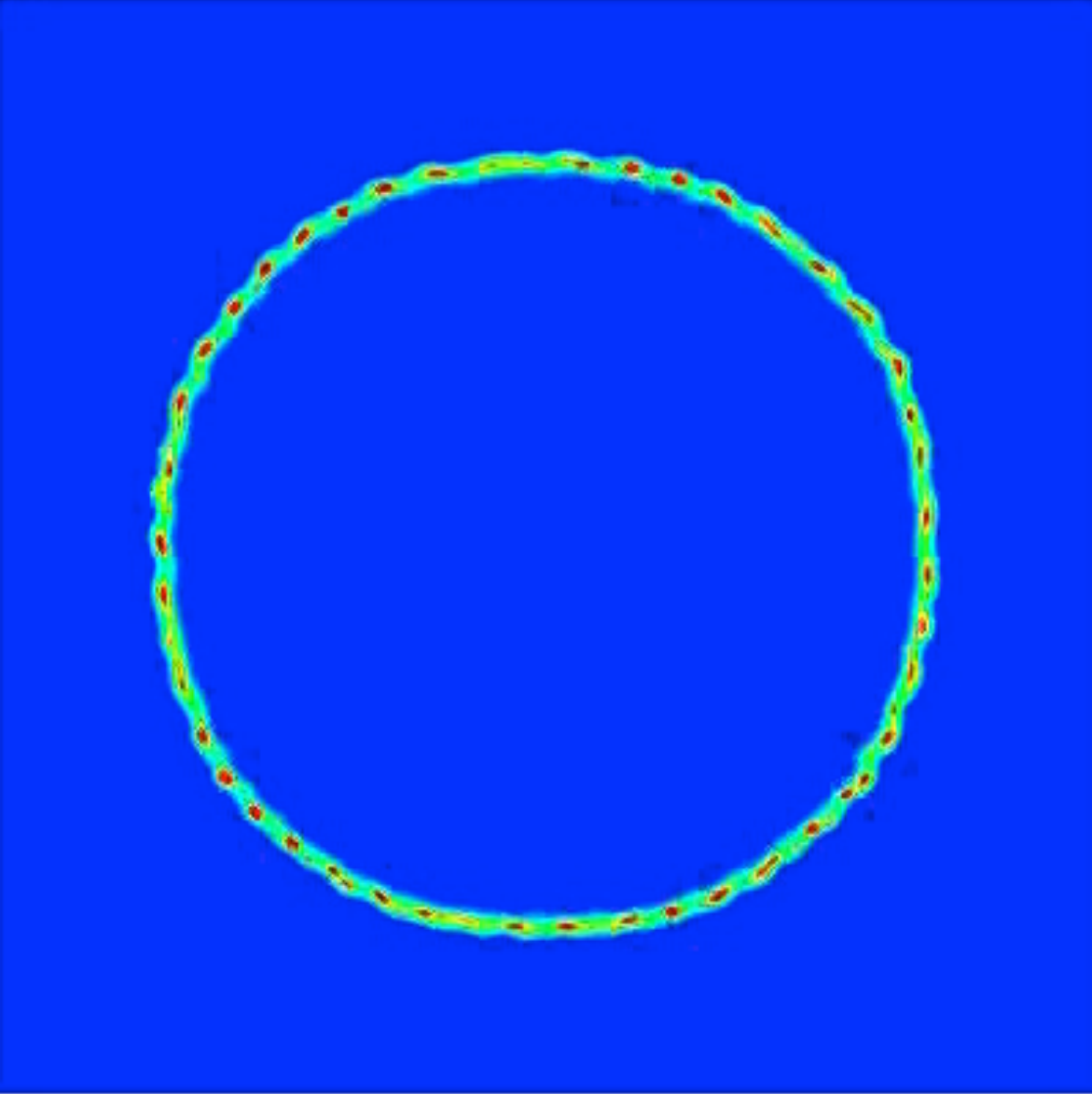}}
 \put(35,40){\includegraphics[scale=0.3]{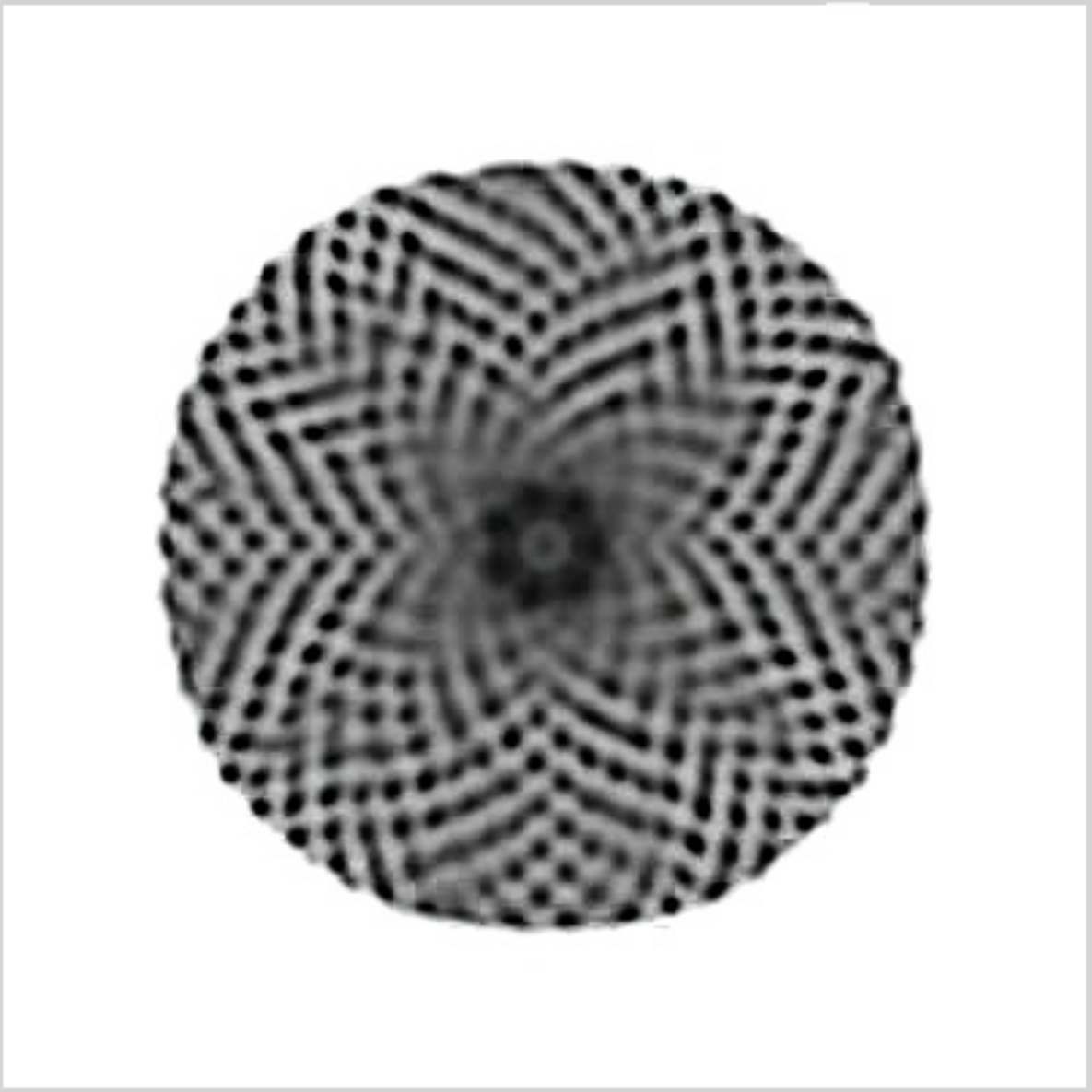}}

 \put(-5,0){\includegraphics[scale=0.3]{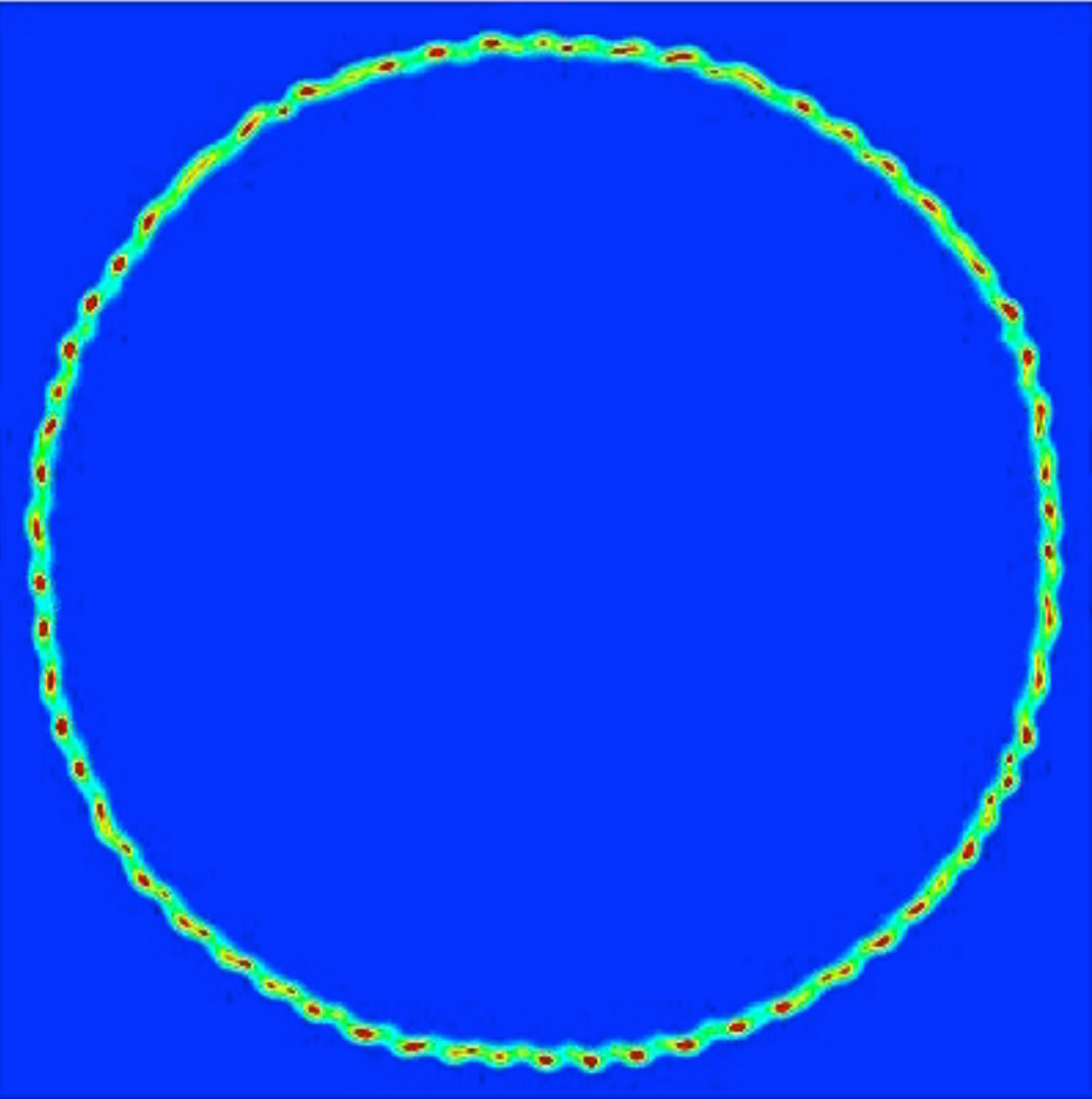}}
 \put(35,0){\includegraphics[scale=0.3]{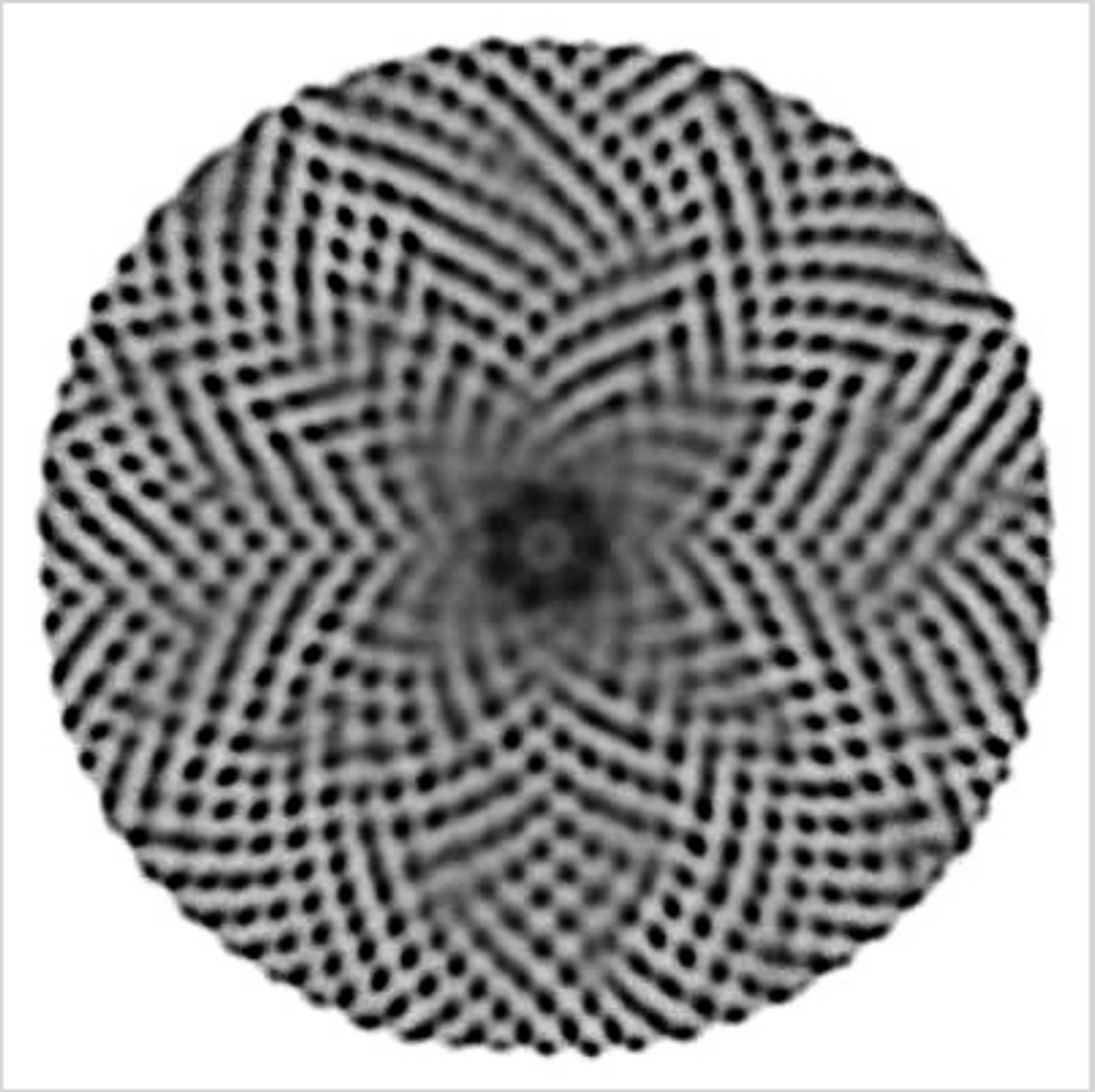}}

 \put(9,198){$u(x,t)$}
\put(42,198){$u(x,t)+w(x,t)$}
\put(-25, 176){$t=1$}
\put(-25, 136){$t=380$}
\put(-25, 96){$t=800$}
\put(-25, 56){$t=1380$}
\put(-25, 16){$t=1840$}
\end{picture}
}
\caption{Numerical simulations of the  densities  $u(x,t)$ and $u(x,t)+w(x,t)$ which satisfy
an initial-boundary value problem for system \eqref{ecoli1i}
 with  $d_c = 10, d_n = 2, \alpha = \beta = \gamma = 1$,
$
g(u) = \frac{1}{2} (1 + \tanh (100 (u - 0.05)),~\chi (c) = \frac{\chi_0 c^2}{c^2 + 0.0625}
$
with $\chi_0 = 0.053$, $b(n)=0.05$, and supplemented with a suitable initial condition concentrated at the origin. 
}
\label{fig:num}
\end{figure}

Unfortunately, our results  do not give any information about a shape of the
limit profile $w_\infty(x)$, as {\it e.g.}~presented on Fig.~\ref{fig:num}.
Numerical simulations presented  on that figure show that, for certain specific functions $g$, $b$ and $\chi$, system \eqref{ecoli1i} describes a formation of chevron-like patterns which are closely related to those observed in biological experiments. Moreover, it is especially remarkable that such systems generate geometrically different patterns depending  on  initial nutrient concentrations, as it was done in  experiments performed by Budrene and Berg.

The aim of this paper is to discuss  properties of solutions to system \eqref{ecoli1i} from a mathematical viewpoint. 
In Theorem \ref{th_inft}, we generalize one-dimensional results from \cite{HMT1} containing and analysis  the one-dimensional initial-boundary value problem for system
\eqref{ecoli1i}  for some specific  functions $g$, $b$ and $\chi$.
Then, we show an analogous theorem  on
 the global-in-time existence and on the large time behavior
of solutions in the space dimensions two and three, under suitable smallness assumptions on  initial conditions, see Theorem~\ref{lp_estim} below.
 We also show that  well-concentrated solutions of a suitable modification
of system \eqref{ecoli1i} may blow up in finite time, see Theorem \ref{blow-up} below for more details.

\subsection*{Notation}
In the sequel, the usual norm of the Lebesgue space $L^p (\Omega)$ with respect to the spatial variable is denoted by $\|\cdot\|_p$ for all $p \in [1,\infty]$ and $W^{k,p}(\Omega)$ is the corresponding Sobolev space with its usual norm defined by
$
\| u \|_{ W^{k,p} } =  \sum_{ |\alpha|\leq k} \| D^\alpha  u\|_p   
$.
%
The letter $C$ corresponds to a generic constant (always independent of $x$ and $t$) which may vary from line to line. Sometimes, we write, $C=C(\alpha,\beta,\gamma, ...)$ when we want to emphasize the dependence of $C$ on parameters~$\alpha,\beta,\gamma, ...$. 

%
\section{Results and comments}

In this work, we prove results on the existence and the large time behavior of solutions to the system
\begin{align}
&u_t= \Delta u-\div(u\nabla\chi(c))+g(u)nu-b(n)u, \label{ecoli1}\\
&c_t=d_c\Delta c+\alpha u-\beta c ,\label{ecoli2}\\
&n_t=d_n\Delta n-\gamma g(u)nu, \label{ecoli3}\\
&w_t= b(n)u, \label{ecoli4}
\end{align}
 considered in a bounded domain
$\Omega\subset\R^d$ with a smooth boundary $\partial\Omega$.
We supplement  these equations  with the Neumann boundary conditions
\begin{align}\label{bound_cond_ecoli}
\partial_\nu  u=\partial_\nu  c=\partial_\nu  n=0
\qquad \text{for}\quad x\in\partial\Omega \quad \text{and}\quad t>0
\end{align}
as well as with  non-negative initial data
\begin{align}
u(x,0)=u_0(x),\quad c(x,0)=c_0(x),\quad  n(x,0)=n_0(x),\quad  w(x,0)=w_0(x). \label{ecoli-ini}
\end{align}
Here, we impose the following 
assumptions on coefficients and on functions which appear in equations
\eqref{ecoli1}--\eqref{ecoli4}.

\begin{assumption} \label{ass}
The diffusion coefficients $d_c>0$ and $d_n>0$ as well as the letters $\alpha>0$, $\beta>0$, $\gamma>0$ in equations \eqref{ecoli2}--\eqref{ecoli4}  denote given constants. Moreover, for
 the functions  $g,b\in C^1([0,\infty))$ and $\chi \in C^2([0,\infty))$, we assume
\begin{itemize}
\item[(i)] $g(0)=0\ $ and $\ g=g(s)\ $ is increasing for $s>0$ and bounded with $G_0\equiv\sup_{s\geq 0}g(s)$;
\item[(ii)] $b(0)=B_0>0\ $ and $\ b=b(s)>0$ is decreasing for $s>0$;
\item[(iii)]  $\chi^\prime, \chi''\in L^\infty([0,\infty))$.
\end{itemize}
\end{assumption}

Under these assumptions, problem \eqref{ecoli1}--\eqref{ecoli-ini} has a unique local-in-time solution
 for all sufficiently regular initial conditions. Moreover, this solution is non-negative if initial conditions \eqref{ecoli-ini} are non-negative. These are more-or-less standard results which we recall in Proposition \ref{thm:local}, below. In this work, we focus mainly on the behavior of non-negative solutions to problem \eqref{ecoli1}--\eqref{ecoli-ini} for large values of time.

{
First, we discuss spatially homogeneous ({\it i.e.}~$x$-independent) non-negative solutions.
By  Proposition \ref{kinetic} below,  such solutions are global-in-time and  converge exponentially towards  constant steady states:
$
(0, 0, \n_\infty, \bar{w}_\infty)
$
 with constants   $\n_\infty\geq 0$ and $\bar{w}_\infty\geq 0$.}
 In Section~\ref{sec-local}, we also study the large time behavior of the mass of a space inhomogeneous solution 
of problem  \eqref{ecoli1}--\eqref{ecoli-ini}
 and we show in Theorem~\ref{conv_mass} below that
it  behaves for large values of time analogously as  space homogeneous solutions, namely, 
\begin{align*}
\left(\int_\Omega u(x,t)\dx, \int_\Omega c(x,t)\dx, \int_\Omega n(x,t)\dx, \int_\Omega w(x,t)\dx\right)
\xrightarrow{t\to\infty}(0,0,\tilde n_\infty, \tilde w_\infty)
\end{align*}
for constants $\tilde n_\infty\geq 0$ and $\tilde w_\infty>0$.

Next, we consider problem \eqref{ecoli1}--\eqref{ecoli-ini} in the one dimensional case and we prove that all solutions corresponding to sufficiently regular, non-negative initial conditions are global-in-time and converge uniformly towards certain steady states. This result has been already proved in \cite{HMT1} for problem \eqref{ecoli1}--\eqref{ecoli-ini} with particular functions $g$, $b$ and $\chi$.
Here, however, we propose a different approach which allows us to consider  general nonlinearities.

\begin{theo}\label{th_inft}
Assume that $d=1$ and $\Omega\subset\R$ is an open and bounded interval.  Let  Assumptions \ref{ass} hold true.
{For every non-negative initial datum  $(u_0,c_0,n_0,w_0)\in  C(\bar \Omega)\times W^{1,p}(\Omega)\times  C(\bar \Omega)\times  C(\bar \Omega)$} with some 
 $p\in\,(1,+\infty)$,
the corresponding solution $(u,c,n,w)$ of problem \eqref{ecoli1}--\eqref{ecoli-ini} exists for all $t>0$ and is non-negative. Moreover, there exists a constant $n_\infty\geq 0$ and a non-negative function $w_\infty\in L^\infty(\Omega)$ such that
\begin{align*}
\big{(}u(x,t),c(x,t),n(x,t),w(x,t)\big{)}\xrightarrow{t\to\infty}(0,0,n_\infty,w_\infty(x))
\end{align*}
exponentially in $L^\infty(\Omega)$.
\end{theo}

An analogous result holds true in higher dimensions, however,  under a smallness assumption imposed on initial conditions.

\begin{theo}\label{lp_estim}
Let $d\in\{2,3\}$ and 
Assumptions \ref{ass} hold true. 
Fix $p_0\in(\frac d 2,\frac d {d-2})$. There exists $\varepsilon(p_0)>0$ such that if $$\max\big(\|u_0\|_{p_0},\|n_0\|_1,\|\nabla c_0\|_{2p_0}\big)<\varepsilon(p_0),$$ then the solution $(u,c,n,w)$ 
of problem \eqref{ecoli1}--\eqref{ecoli-ini}
exists for all $t>0$.
Moreover, there exist a constant $n_\infty\geq 0$ and a non-negative function $w_\infty\in L^\infty(\Omega)$ such that
\begin{align}\label{conv_mult}
\big{(}u(x,t),c(x,t),n(x,t),w(x,t)\big{)}\xrightarrow{t\to\infty}(0,0,n_\infty,w_\infty(x))
\end{align}
exponentially in $L^\infty(\Omega)$.
\end{theo}

Here, 
we have to limit ourselves to the dimension $d\in\{2,3\}$ (note that the interval $(\frac{d}{2},\frac{d}{d-2})$ is nonempty only in this case)
because of methods used in the proof of Theorem~\ref{lp_estim}. 
 Obviously, this is not an important  constraint from a point of view of applications. 

There is an immediate  question if the smallness assumption in Theorem~\ref{lp_estim}  is indeed necessary to show both results:
the global-in-time existence of non-negative solutions to problem \eqref{ecoli1}--\eqref{ecoli-ini} and their exponential convergence toward steady states as in \eqref{conv_mult}.
It seems that  this is  indeed the case 
in view of blowup results obtained for 
 the so-called parabolic-elliptic Keller-Segel model
of chemotaxis, see {\it e.g.}~\cite{JL92, N01},
as well as for its parabolic-parabolic counterpart 
\eqref{KS0}, see \cite{W4,M20} and the references therein.

In our next result,  we use an idea which is well-known in the study of the blowup phenomenon for 
 the parabolic-elliptic 
Keller-Segel model.
Thus,  we  consider solutions to a modified problem 
\eqref{ecoli1}--\eqref{ecoli-ini}, 
where the parabolic equation \eqref{ecoli2}  for $c=c(x,t)$ is replaced by its elliptic counterpart:
\begin{align}
&u_t= \Delta u-\div(u\nabla \chi(c))+g(u)nu-b(n)u, \label{eq1-b}\\
&0=\Delta c+ \alpha u- \beta c, \label{eq2-b}\\
&n_t=d_n\Delta n-\gamma g(u)nu,\label{eq3-b}
\end{align}
in a bounded domain $\Omega\subset\R^d$, supplemented with the Neumann boundary conditions
\begin{align}\label{bound_cond-b}
\partial_\nu  u=\partial_\nu  c=\partial_\nu  n=0\quad\text{for}\quad x\in\partial\Omega\quad\text{and}\quad t>0,
\end{align}
and with non-negative initial data
\begin{align}
u(x,0)=u_0(x),  \quad n(x,0)=n_0(x). \label{eq-ini-b}
\end{align}
Here, we have omitted the equation for $w=w(x,t)$ because this quantity can be obtained from other variables in the way explained above.
{Let us briefly review preliminary results on this initial boundary-value problem.

\begin{itemize}

\item
Under Assumptions \ref{ass},
for every non-negative initial datum   $(u_0,c_0,n_0)\in  C(\bar \Omega)\times W^{1,p}(\Omega)\times  C(\bar \Omega)$ 
with fixed $p\in\,(d,+\infty)$
there exists a unique non-negative local-in-time solution $u=u(x,t)$, $c=c(x,t)$, $n=n(x,t)$ of problem \eqref{eq1-b}--\eqref{eq-ini-b}. Here, it suffices to follow ideas from the proof of Proposition \ref{thm:local}.

\item Under Assumptions \ref{ass} and  if $d=1$,  a  solution to problem  \eqref{eq1-b}--\eqref{eq-ini-b} corresponding to a sufficiently regular nonnegative  initial condition 
is global-in-time and there exists a constant $n_\infty\geq 0$ such that
		\begin{align*}
		\big{(}u(x,t),c(x,t),n(x,t)\big{)}\xrightarrow{t\to\infty}(0,0,n_\infty)
		\end{align*}
		exponentially in $L^\infty(\Omega)$. The proof of this result requires a minor modification of arguments from the proof of Theorem \ref{th_inft}.

\item An analogous asymptotic result holds true  if $d\in\{2,3\}$  under a suitable smallness assumption imposed on initial data --  as those in Theorem \ref{lp_estim}.

	\end{itemize}
}

In the following theorem, we show that 
sufficiently well-concentrated solutions of problem \eqref{eq1-b}--\eqref{eq-ini-b}
in two dimensions 
 cannot be extended to all $t>0$  and  here, we follow classical ideas by Nagai \cite{N01}.
Presenting this particular result, we want to emphasize  that several blowup results obtained for the parabolic-elliptic model of chemotaxis can be directly applied to the model from this work.

\begin{theo}\label{blow-up}
Let $d=2$ and Assumptions \ref{ass} hold true with $\chi(c)=\chi_0 c$ for some $\chi_0>0$.
Consider a nonnegative local-in-time solution $(u,c,n)$ of problem \eqref{eq1-b}--\eqref{eq-ini-b} corresponding to an initial datum 
$(u_0,n_0)\in C(\bar \Omega)\times C(\bar \Omega)$.
 Assume that
\begin{align*}
M_0=\int_\Omega u_0(x)\dx>\frac{8\pi}{\alpha\chi_0}.
\end{align*}
For each $q\in \Omega$ there exists $\varepsilon(q)>0$ such
that if \[\int_{\Omega}u_0(x)|x-q|^2\dx<\varepsilon(q)\]
then the solution $(u,c,n)$ cannot be extended to all $t>0$.
\end{theo}

This theorem is proved in Section \ref{sec:blow}.
 {Here, let us only remark that  the  linear ``death'' term 
 $-b(n)u$ in equation \eqref{eq1-b}
 is not strong enough  to prevent a blow-up of solutions in a finite time. It seems that even some superlinear death terms fail to ensure 
 the existence of global-in-time solutions as stated in \cite{W3}.}

\section{Preliminary results on existence and large time behavior of solutions}\label{sec-local}


First, we present briefly a result on an existence of local-in-time solutions to the considered initial-boundary value problem.

%

 \begin{prop}\label{thm:local}
 Let Assumptions \ref{ass} hold true. Fix $p\in\,(d,+\infty)$. For every non-negative initial datum   $(u_0,c_0,n_0,w_0)\in  C(\bar \Omega)\times W^{1,p}(\Omega)\times  C(\bar \Omega)\times  C(\bar \Omega)$ there exists a maximal time of existence $T_{\mathrm{max}}\in(0,+\infty]$ and a non-negative unique classical solution 
 $u=u(x,t)$, $c=c(x,t)$, $n=n(x,t)$ , $w=w(x,t))$
 of problem  \eqref{ecoli1}--\eqref{ecoli-ini}
 such that
 \[ \begin{array}{l} u,w \in C\big(\bar \Omega \times [0,T_{\mathrm{max}})\big)\cap  C^{2,1}\big(\bar \Omega \times (0,T_{\mathrm{max}})\big),\\
 c\in C\big(\bar \Omega \times [0,T_{\mathrm{max}})\big)\cap  C^{2,1}\big(\bar \Omega \times (0,T_{\mathrm{max}})\big)\cap L^{\infty}_\mathrm{loc}( (0,T_{\mathrm{max}}); W^{1,p}(\Omega)),\\
 n\in  C\big(\bar \Omega \times [0,T_{\mathrm{max}})\big)\cap  C^{2,1}\big(\bar \Omega \times (0,T_{\mathrm{max}})\big)\cap L^{\infty}( \Omega \times (0,T_{\mathrm{max}})). \end{array} \]
 Moreover, if $T_{\mathrm{max}}<+\infty$, then
 \begin{equation}\label{BlowUpNorm}
 	\limsup_{t\to T_{\mathrm{max}}} \left(\ \| u(t) \|_\infty + \| c(t) \|_{W^{1,p}}\ \right) = +\infty.
 \end{equation}
  \end{prop}
  
 A  local-in-time solution  in Proposition~\ref{thm:local} 
 can be obtained via the Banach fixed point argument applied to the the following Duhamel formulation of the problem \eqref{ecoli1}--\eqref{ecoli-ini} 
\begin{align}\label{duh1}
\begin{split}
u(t)=&\ {\rm e}^{\Delta t}u_0-\int_0^t{\rm e}^{\Delta(t-s)}\nabla\cdot \big(u(s)\nabla\chi(c(s))\big)\dy{s}\\
&+\int_0^t{\rm e}^{\Delta(t-s)}u(s)(g(u)n-b(n))(s)\dy{s},
\end{split}
\\
c(t)=&\ {\rm e}^{(d_c\Delta-\beta)t}c_0+\alpha\int_0^t{\rm e}^{(d_c\Delta-\beta)(t-s)}u(s)\dy{s},\label{duh2}\\
n(t)=&\ {\rm e}^{d_n\Delta t}n_0-\gamma\int_0^t{\rm e}^{d_n\Delta(t-s)}g(u(s))n(s)u(s)\dy{s}.\label{duh3}
\end{align}
Here, the symbol $\{e^{\Delta t}\}_{t\geq 0}$ denotes the semigroup of linear operators on $L^p(\Omega)$ 
generated by Laplacian with the Neumann boundary conditions
(we recall some estimates of this semigroup in Lemma \ref{lem:LpLq} below). 
A regularity of such a solution is shown by standard bootstrapping arguments. A positivity of the functions  $u(x,t)$, $c(x,t)$, $n(x,t)$  as well as the estimate 
\begin{equation}\label{n:est0}
0\leq n(x,t)\leq \|n_0\|_\infty\quad \text{for all}\quad x\in\Omega,\ t\in[0,T],
\end{equation}
(as long as the term $\gamma g(u)nu$ is nonnegative)
 are a natural consequence of the maximum principle.
 Finally, the function  $w(x,t)$ is recovered as an integral of the other quantities, see equation \eqref{w:0}. 
We skip the detailed proof of Proposition \ref{thm:local} because it can be completed by a  straightforward adaptation of methods for previous works on an initial-boundary value problem for the Keller-Segel model. For details of such a reasoning,  we refer the reader to 
the monograph \cite{yagi} 
as well as to
the papers
\cite[Theorem 3.1]{HW}, 
\cite[Lemma 3.1]{BBTW},  
and to references therein.

Next, we discuss spatially homogeneous non-negative solutions.
Notice that if an initial condition \eqref{ecoli-ini} is independent of $x$, namely, if
\begin{equation}\label{ini-const}
u_0(x)=\u_0,\; c_0(x)=\bar{c}_0,\; \n_0(x)=\n_0,\; \bar{w}_0(x)=\bar{w}_0
\end{equation}
for some  constants $ \u_0, \bar{c}_0, \n_0, \bar{w}_0\in[0,\infty)$,
then the corresponding solution 
 of problem \eqref{ecoli1}--\eqref{ecoli-ini}
is also independent of $x$  which  is an immediate consequence of the uniqueness of solutions established in Proposition \ref{thm:local}.
Let us formulate  a result  on a large time behavior of  such non-negative space homogeneous solutions to problem \eqref{ecoli1}--\eqref{ecoli-ini}.
{
\begin{prop}\label{kinetic}
Let  Assumptions \ref{ass}
be satisfied. For every non-negative, constant initial condition  \eqref{ini-const}, the corresponding solution $\big{(}\u(t), \bar{c}(t), \n(t), \bar{w}(t)\big{)}$ to problem \eqref{ecoli1}--\eqref{ecoli-ini} is $x$-independent, 
non-negative, global-in-time, and converges exponentially towards the constant steady state  $(0, 0, \n_\infty, \bar{w}_\infty)$ for some $\n_\infty\geq 0$ and $\bar{w}_\infty\geq 0$ depending on initial conditions.
\end{prop}

\begin{proof}[Sketch of the proof.]
Obviously, the chemotactic term $-\div(u\nabla\chi(c))$ as well as the terms in equations \eqref{ecoli1}--\eqref{ecoli4} containing Laplacian disappear in the case of $x$-independent solutions and we obtain  
the following system of the corresponding ordinary differential equations
\begin{align}
&\frac{d}{\dy{t}}\u= g(\u)\n\u-b(\n)\u  \label{eq1s}\\
&\frac{d}{\dy{t}}\bar{c}=\alpha \u-\beta \bar{c}\label{eq2s}\\
&\frac{d}{\dy{t}}\n=-\gamma g(\u)\n\u \label{eq3s}\\
&\frac{d}{\dy{t}}\bar{w}= b(\n)\u. \label{eq4s}
\end{align}

It is clear that it suffices to consider only 
equations \eqref{eq1s} and \eqref{eq3s} for the functions 
$\u(t)$ and $\n(t)$. This system of two equations has the constant steady state   $(0, \overline{n}_\infty)$ for each constant  $ \overline{n}_\infty\geq 0$.
Applying a routine phase portrait analysis one can show 
 every solution $(\u(t),\n(t))$ of equations \eqref{eq1s} and 
 \eqref{eq3s} which starts in the first quadrant $(\u>0, \n>0)$ at $t=0$ has to remain in this quadrant of the $(\u,\n)$-plane for all future times and converges exponentially towards $(0, \overline{n}_\infty)$.
 We skip other details of this proof because they are analogous to those in the proof of Theorem \ref{conv_mass}, below. 
 \end{proof}
}

Now, we consider solutions of problem \eqref{ecoli1}--\eqref{ecoli-ini} with nonconstant initial conditions and we prove a result analogous to the one in Proposition \ref{kinetic} on the large time behavior of the integrals $\int_\Omega u(x,t)\dx$, $\int_\Omega c(x,t)\dx$, $\int_\Omega n(x,t)\dx$ and $\int_\Omega w(x,t)\dx$.

\begin{theo}\label{conv_mass}
Assume that a non-negative solution $(u,c,n,w)$ of problem \eqref{ecoli1}--\eqref{ecoli-ini} exists for all $t>0$. Let   Assumptions \ref{ass}
 hold true. Then
\begin{align*}
\int_\Omega u(t)\dx\rightarrow 0\quad\text{and}\quad \int_\Omega c(t)\dx\rightarrow 0\qquad \text{as}\ \ t\to\infty,
\end{align*}
and there are constants $\tilde{n}_\infty\geq 0$ and $\tilde{w}_\infty>0$ such that
\begin{align*}
\int_\Omega n(t)\dx\rightarrow \tilde{n}_\infty\quad\text{and}\quad \int_\Omega w(t)\dx\rightarrow \tilde{w}_\infty\qquad \text{as}\ \ t\to\infty.
\end{align*}
\end{theo}

\begin{proof}
First, integrating equations \eqref{ecoli1}--\eqref{ecoli4} with respect to $x$, we obtain
\begin{align}
&\frac{d}{dt}\int_\Omega u\dx=\int_\Omega g(u)nu\dx-\int_\Omega b(n)u\dx \label{int_eq1}\\
&\frac{d}{dt}\int_\Omega c\dx=\alpha\int_\Omega u\dx-\beta\int_\Omega c\dx \label{int_eq2}\\
&\frac{d}{dt}\int_\Omega n\dx=-\gamma\int_\Omega g(u)nu\dx \label{int_eq3}\\
&\frac{d}{dt}\int_\Omega w\dx=\int_\Omega b(n)u\dx. \label{int_eq4}
\end{align}
Since,
$
\frac{d}{dt}\left(\int_\Omega u(t)\dx+\frac{1}{\gamma}\int_\Omega n(t)\dx+ \int_\Omega w(t)\dx\right)=0,
$
we get the conservation of mass in the following sense
\begin{align}\label{total-mass}
\int_\Omega u(t)\dx+\frac{1}{\gamma}\int_\Omega n(t)\dx+\int_\Omega w(t)\dx= \int_\Omega u_0\dx+\frac{1}{\gamma}\int_\Omega n_0\dx+\int_\Omega w_0\dx,
\end{align}
for all $t>0$. In particular, since all functions are non-negative, we have
\begin{align}\label{w-mass}
\int_\Omega u(x,t)\dx\leq \int_\Omega u_0\dx+\frac{1}{\gamma}\int_\Omega n_0\dx+\int_\Omega w_0\dx\quad\text{for all}\quad t>0.
\end{align}
Now, we improve this estimate by adding equation \eqref{int_eq1} to equation \eqref{int_eq3} multiplied by $\gamma^{-1}$ and integrating resulting equation over $[0,t]$ to obtain the relation
\begin{align}\label{teo-eq1}
\int_\Omega u(t)+\gamma^{-1}n(t)\dx=\int_\Omega u_0+\gamma^{-1}n_0\dx-\int_0^t\int_\Omega b(n(s))u(s)\dx\dy{s}
\end{align}
which by positivity of $b$ and $u$ implies 
\begin{align}\label{mass-ecoli}
\int_\Omega u(t)\dx\leq \int_\Omega u_0\dx+\frac{1}{\gamma}\int_\Omega n_0\dx.
\end{align}

Next, we observe that, since $g(u)nu\geq 0$, it follows from equation \eqref{int_eq3} that the integral $\int_\Omega n(t)\dx$ is nonincreasing in $t$ and since it is also non-negative, the following finite limit exists
\begin{align}\label{lim}
\lim_{t\to\infty}\int_\Omega n(t)\dx=\tilde{n}_\infty\geq 0.
\end{align}

Now, since $b(n)u\geq 0$, equation \eqref{teo-eq1} implies that the mapping $t\mapsto\int_\Omega u(t)+\gamma^{-1}n(t)\dx$ is also nonincreasing, hence, it has a limit as $t\to\infty$. Consequently, 
using relations \eqref{lim}, we conclude that there exists a constant $\tilde{u}_\infty\geq 0$ such that
$
\lim_{t\to\infty}\int_\Omega u(t)\dx=\tilde{u}_\infty.
$
Moreover, since $\int_\Omega u(t)+\gamma^{-1}n(t)\dx$ is bounded for $t\geq 0$, identity \eqref{teo-eq1} implies that
$b(n)u\in L^1\big{(}(0,\infty);L^1(\Omega)\big{)}$. However, since $b(n)\geq b(\|n_0\|_\infty)>0$, it follows that
\begin{align}\label{u-l1}
u\in L^1((0,\infty);L^1(\Omega)).
\end{align}
Consequently, we have $\tilde{u}_\infty=0$.
Since $b(n)u\in L^1\big{(}(0,\infty), L^1(\Omega)\big{)}$ we obtain from equation \eqref{int_eq4}
\begin{align*}
\lim_{t\to\infty}\int_\Omega w(t)\dx=\int_\Omega w_0\dx+\int_0^\infty\int_\Omega b(n)u\dx\dy{s}\equiv \tilde{w}_\infty>0.
\end{align*}
Finally, $\lim_{t\to\infty}\int_\Omega c(t)\dx=0$ due to equation \eqref{int_eq2} because $\lim_{t\to\infty}\|u(t)\|_1=0$. 
\end{proof}


\section{Problem in one space dimension}
The proof of Theorem \ref{th_inft} requires the following two auxiliary lemmas. First, we find an estimate of $c_x(t)$ which is uniform in time.

\begin{lemm}\label{cx_to_0}
Let the assumptions of Theorem \ref{th_inft} hold true and denote by (u,c,n,w) the corresponding non-negative local-in-time solution to problem \eqref{ecoli1}--\eqref{ecoli-ini} on $[0,T_{\mathrm{max}})$ 
constructed in Proposition  \ref{thm:local}.
 For each $p\in[2,\infty)$ there exists a constant $C=C(p)>0$ independent of $T_{\mathrm{max}}$ such that
$
\|c_x(t)\|_p\leq C
$
for all $t\in(0,T_{\mathrm{max}})$. Moreover, if the solution is global-in-time, then
$
\lim_{t\to\infty}\|c_x(t)\|_p=0
$
for each $p\in[2,\infty)$.
\end{lemm}
\begin{proof}
Using the Duhamel formula \eqref{duh2} and the estimates of the heat semigroup \eqref{G5}, \eqref{G3} we obtain
\begin{equation}\label{cx:dec}
\begin{split}
\|c_x(t)\|_p&\leq\|\partial_x {\rm e}^{t\Delta-\beta t}c_0\|_p+ \alpha\int_0^t\|\partial_x {\rm e}^{(\Delta-\beta)(t-s)}u(s)\|_p\dy{s}\\
&\leq C{\rm e}^{-\beta t}\|c_{0,x}\|_p+ C\int_0^t\left(1+(t-s)^{-\frac{1}{2}(1-\frac{1}{p})-\frac{1}{2}}\right){\rm e}^{-(\beta+\lambda_1)(t-s)}\|u(s)\|_1\dy{s}
\end{split}
\end{equation}
for all $t\in(0,T]$ and a constant $C>0$ independent of $t>0$. The right-hand side of this inequality is bounded uniformly in $t>0$ and independent of $T>0$ because of estimate \eqref{mass-ecoli}. 
Moreover, if the solution is global-in-time, it converges to zero
by  Lemma \ref{serre} below, since $\lim_{t\to\infty}\|u(t)\|_1=0$ by Theorem \ref{conv_mass}.
\end{proof}

Next, we show the boundedness of the $L^2$-norm of $u$ using  energy estimates. 

\begin{lemm}\label{L2-bound}
Let the assumptions of Theorem \ref{th_inft} hold true. Moreover, let $(u,c,n,w)$ be the non-negative local-in-time solution to problem \eqref{ecoli1}--\eqref{ecoli-ini} constructed in Proposition~\ref{thm:local}. Then, there exists a numeber $C>0$ independent of $T_{\mathrm{max}}$ such that
$
\|u(t)\|_2\leq C
$
for all $t\in[0,T_{\mathrm{max}})$.
\end{lemm}

\begin{proof}
After multiplying equation \eqref{ecoli1} by $u$ and integrating over $\Omega$ we obtain
\begin{equation*}
\frac{1}{2}\frac{d}{\dy{t}}\int_\Omega u^2\dx+\int_\Omega (u_x)^2\dx+\int_\Omega b(n)u^2\dx=\int_\Omega g(u)nu^2\dx+\int_\Omega uc_x\chi^\prime(c)u_x\dx.
\end{equation*}
Thus, by the Cauchy inequality and  Assumptions \ref{ass}, we get
\begin{equation}\label{l1-e1}
\begin{split}
\frac{1}{2}\frac{d}{\dy{t}}\int_\Omega u^2\dx&+\frac{1}{2}\int_\Omega (u_x)^2\dx+b(\|n_0\|_\infty)\int_\Omega u^2\dx\\
&\leq G_0\|n_0\|_\infty\int_\Omega u^2\dx+ \frac{\|\chi^\prime\|_{L^\infty(\R)}^2}{2}\int_\Omega u^2(c_x)^2\dx,
\end{split}
\end{equation}
where constants $b(\|n_0\|_\infty)=\inf_{n}b(n)>0$ and $G_0=\sup_{u}g(u)>0$ are finite.
To deal with the last term on the right-hand side of \eqref{l1-e1} we use estimate \eqref{mass-ecoli} and Lemma \ref{cx_to_0} combined with the H\" older, Sobolev, 
and the $\varepsilon$-Cauchy inequalities in the following way
\begin{align*}
\int_\Omega u^2(c_x)^2\dx\leq\|u\|_4^2\|c_x\|_4^2\leq C\|u\|_{W^{1,2}}\|u\|_1\|c_x\|_4^2\leq \varepsilon\|u\|_{W^{1,2}}^2+C(\varepsilon),
\end{align*}
where the quantity $C(\varepsilon)=C(\varepsilon,\|u(t)\|_1,\|c_x(t)\|_4)$ is uniformly bounded in $t$ by inequality \eqref{mass-ecoli} and Lemma \ref{cx_to_0}.
Moreover, by the Sobolev inequality and the Young inequality, 
\begin{equation}\label{l1-e4}
\int_\Omega u^2\dx\leq C\|u\|_{W^{1,2}}^{2/3}\|u\|_1^{4/3}\leq \varepsilon\|u\|_{W^{1,2}}^2+C_\varepsilon\|u\|_1^2.
\end{equation}
Therefore, for every $\varepsilon>0$ there is a constant $C(\varepsilon)>0$ such that
\begin{equation}\label{l1-e2}
\frac{1}{2}\frac{d}{\dy{t}}\int_\Omega u^2\dx+\frac{1}{2}\int_\Omega (u_x)^2\dx+b(\|n_0\|_\infty)\int_\Omega u^2\dx\leq \varepsilon\|u\|_{W^{1,2}}^2+C(\varepsilon).
\end{equation}
The term on the right-hand side of \eqref{l1-e2} containing small $\varepsilon>0$ can be absorbed by the corresponding two terms on the left-hand side. Thus, we obtain the following differential inequality
\begin{align*}
\frac{d}{\dy{t}}\int_\Omega u^2\dx+C\|u\|^2_{W^{1,2}}\leq C,
\end{align*}
with a constant $C>0$ which, in particular, implies that $\|u(t)\|_2$ has to be bounded uniformly in $t$.
\end{proof}

The reminder of this section is devoted to the proof of Theorem \ref{th_inft} on the large time behavior of solutions to problem \eqref{ecoli1}--\eqref{ecoli-ini} in a one dimensional domain.

\begin{proof}[Proof of Theorem \ref{th_inft}]
Local-in-time solutions constructed in Theorem  \ref{thm:local} can be extended to all $t>0$ due to 
relations \eqref{BlowUpNorm} and 
{\it a priori} estimates which will be obtained below in the study of their large time behavior. We skip this standard reasoning and we proceed directly to estimates of solutions for large values of $t>0$.

{\it Step 1: $\lim_{t\to\infty}\|u(t)\|_\infty=0$.}
We apply the Duhamel formula \eqref{duh1} in the following way
\begin{equation}\label{u-v}
\begin{split}
\Big{\|}u(t)-{\rm e}^{\Delta t}u_0&- \int_0^t{\rm e}^{\Delta(t-s)}u(s)(g(u)n-b(n))(s)\dy{s}\Big{\|}_\infty\\ &=\Big{\|}\int_0^t{\rm e}^{\Delta(t-s)}\partial_x\big(u(s)c_x(s)\chi^\prime(c)\big)\dy{s}\Big{\|}_\infty.
\end{split}
\end{equation}
Using the property of the heat semigroup from Lemma \ref{lem:LpLq} below, the H\" older inequality, and  Assumptions \ref{ass} on the function $\chi$ we estimate the right-hand side of equation \eqref{u-v} as follows
\begin{equation}\label{t1-e1}
\begin{split}
\Big\|\int_0^t{\rm e}^{\Delta(t-s)}&\partial_x\big(u(s)c_x(s)\chi^\prime(c)\big)\dy{s}\Big\|_\infty\\
&\leq C\|\chi^\prime\|_\infty\int_0^t\left(1+(t-s)^{-\frac{5}{6}}\right){\rm e}^{-\lambda_1(t-s)} \|u(s)c_x(s)\|_{3/2}\dy{s}\\
&\leq C\|\chi^\prime\|_\infty\int_0^t\left(1+(t-s)^{-\frac{5}{6}}\right){\rm e}^{-\lambda_1(t-s)}\|u(s)\|_2\|c_x(s)\|_6\dy{s}.
\end{split}
\end{equation}

Thus, by Lemma \ref{serre} below, the integral on the right-hand side of
inequality  \eqref{t1-e1} tends to zero because $\|u(t)\|_2$ is bounded by Lemma~\ref{L2-bound} and because $\|c_x(t)\|_6$ tends to zero which is proved in Lemma~\ref{cx_to_0}. Hence, coming back to identity \eqref{u-v}, we see that
\begin{align}\label{th_1d1}
\|u(t)-v(t)\|_\infty\rightarrow 0 \quad \text{as}\quad t\to\infty,
\end{align}
where $v(x,t)$ is a solution to the problem
\begin{align}
&v_t=v_{xx}+g(u)nu-b(n)u,\label{rd1}\\
&v(x,0)=u_0(x),\label{rd1-ini}
\end{align}
supplemented with the Neumann boundary conditions. We denote the nonlinear term on the right-hand side of \eqref{rd1} by $f\equiv g(u)nu-b(n)u$ and since $g$, $b$ and $n$ are bounded, there exist a constant $C>0$ such that
\begin{align*}
\|f(\cdot,t)\|_1\leq C\|u(t)\|_1\to 0\quad \text{as}\quad t\to\infty
\end{align*}
by Theorem \ref{conv_mass}. Hence, by Lemma \ref{lemm_rd} we obtain
\begin{align}\label{st1-1}
\Big{\|}v(t)-\frac{1}{|\Omega|}\int_\Omega v(t)\dx\Big{\|}_\infty\to 0\quad \text{as}\quad t\to\infty.
\end{align}
However, integrating equation \eqref{rd1} with respect to $x$ and comparing 
the resulting formula 
 with equation \eqref{int_eq1}, it is easy to see that $\int_\Omega u(t)\dx=\int_\Omega v(t)\dx$ for all $t>0$. Therefore, using \eqref{th_1d1} and \eqref{st1-1} we obtain the convergence
\begin{align*}
\Big{\|}u(t)-\frac{1}{|\Omega|}\int_\Omega u(t)\dx\Big{\|}_\infty\leq \|u(t)-v(t)\|_\infty+ \Big{\|}v(t)-\frac{1}{|\Omega|}\int_\Omega v(t)\dx\Big{\|}_\infty\to 0
\end{align*}
as $t\to\infty$  which, in virtue of Theorem \ref{conv_mass}, completes the proof that $\lim_{t\to\infty}\|u(t)\|_\infty=0$.

{\it Step 2: Exponential decay of $\int_\Omega u(t)\dx$.} Recall that the function $b(n(x,t))$ is bounded from below by $b(\|n_0\|_\infty)>0$ because $b$ is 
nonincreasing, {\it cf.}  Assumptions \ref{ass}. Hence, since $\|u(t)\|_\infty\to 0$ as $t\to\infty$ and since $g(0)=0$, there exist constants $T>0$ and $\mu>0$ such that for all $t\geq T$ and all $x\in\Omega$ we have
\begin{equation*}
\big{(}g(u)n-b(n)\big{)}(x,t)\leq -\mu.
\end{equation*}
Thus, using equation \eqref{int_eq1} we get the following differential inequality
\begin{equation*}
\frac{d}{\dy{t}}\int_\Omega u(t)\dx\leq -\mu\int_\Omega u(t)\dx,
\end{equation*}
which implies the exponential decay
\begin{align}\label{L1:dec}
\|u(t)\|_1\leq\|u_0\|_1{\rm e}^{-\mu t}\quad \text{for all}\quad t>0.
\end{align}

Now, we use this estimate to improve Lemma \ref{cx_to_0}.

{\it Step 3: Exponential decay of $\|c_x(t)\|_p$ for each $p\in[1, \infty)$.}
Using the exponential decay of $\|u(t)\|_1$ from inequality  \eqref{L1:dec} in estimate \eqref{cx:dec} and Lemma \ref{lem:LpLq}, we obtain
\begin{align*}
\|c_x\|_p\leq C{\rm e}^{-\beta t}\|c_{0,x}\|_p+ C\int_0^t\left(1+(t-s)^{-\frac{1}{2}\left(1-\frac{1}{p}\right)-\frac{1}{2}}\right){\rm e}^{-(\beta+\lambda_1)(t-s)}{\rm e}^{-\mu s}\dy{s},
\end{align*}
where the integral on the right-hand side decays exponentially by Lemma \ref{serre}.

{\it Step 4: Exponential decay of $\|c(t)\|_\infty$.}
Applying the Duhamel principle \eqref{duh2}, computing the $L^\infty$-norm, and using the heat semigroup estimate \eqref{G2} we have
\begin{align*}
\|c(t)\|_\infty\leq C{\rm e}^{-\beta t}\|c_0\|_\infty+ C\int_0^t\left(1+(t-s)^{-\frac{1}{2}}\right){\rm e}^{-\beta(t-s)}\|u(s)\|_1\dy{s}
\end{align*}
for all $t>0$ and a constant $C>0$ independent of $t>0$. Since $\|u(t)\|_1$ decays exponentially, see \eqref{L1:dec}, we complete the proof of this step by Lemma \ref{serre}, again.

{\it Step 5: Exponential decay of $\|u(t)\|_\infty$.}
Here, it suffices to repeat all the estimates from Step~1 using the exponential decay estimates of $\|c_x(t)\|_6$ established in Step 3 and the decay of $\|u(t)\|_1$ from Step~2.

{\it Step 6: Exponential convergence $\lim_{t\to\infty}\|n(t)-n_\infty\|_\infty=0$.}
By Theorem \ref{conv_mass}, the limit
\begin{align*}
\lim_{t\to\infty}\int_\Omega n(t)\dx\equiv \tilde{n}_\infty=\int_\Omega n_0\dx- \int_0^\infty\int_\Omega \gamma g(u)nu\dx\dy{s}
\end{align*}
exists and is non-negative. This is, in fact, exponential convergence, because by equation \eqref{int_eq3} and by Step 2 we have
\begin{equation*}
\left|\int_\Omega n(t)\dx-\tilde{n}_\infty\right| \leq\gamma\int_t^\infty\int_\Omega|g(u)nu|\dx\dy{s}
\leq\gamma G_0\|n_0\|_\infty\int_t^\infty \|u(s)\|_1\dy{s} \leq C{\rm e}^{-\mu t}.
\end{equation*}
Now, applying Lemma \ref{lemm_rd} with $f(x,t)=-\gamma g(u)nu$ to equation \eqref{ecoli3}, since $\|f(\cdot,t)\|_1\to 0$ exponentially as $t\to\infty$, we obtain
\begin{align*}
\Big{\|}n(t)-\frac{1}{|\Omega|}\int_\Omega n(t)\dx\Big{\|}_\infty\to 0\quad \text{exponentially as}\quad t\to\infty.
\end{align*}
Combining these two convergence results we complete the proof of Step 6 with $n_\infty=|\Omega|^{-1}\tilde{n}_\infty$.

{\it Step 7: $\|w(t)-w_\infty\|_\infty\to 0$ exponentially as $t\to\infty$.}
Here, we define
\begin{align}\label{w-inft}
w_\infty(x)=w_0(x)+\int_0^\infty b(n(x,t))u(x,t)\dy{t}.
\end{align}
Notice, that since $b$ is bounded and $\|u(t)\|_\infty$ decays exponentially, the right-hand side of \eqref{w-inft} belongs to $L^\infty(\Omega)$. Moreover, it is easy to see that
\begin{align*}
\|w(t)-w_\infty(x)\|_\infty&=\left\|\int_t^\infty b(n)u(x,s)\dy{s}\right\|_\infty \leq C\int_t^\infty\|u(s)\|_\infty\dy{s}\\
&\leq C\int_t^\infty {\rm e}^{-\mu s}\dy{s}\to 0
\end{align*}
exponentially as $t\to\infty$. This completes the proof of Step 7 and of Theorem~\ref{th_inft}.
\end{proof}

\section{Problem in higher dimensions}\label{sec:high}

\begin{proof}[Proof of Theorem \ref{lp_estim}]
As in the one dimensional case, we consider the unique non-negative local-in-time solution to problem \eqref{ecoli1}--\eqref{ecoli-ini} which is constructed in Proposition \ref{thm:local}. This solution can be continued to the global one due to estimates proved below (see also relations \eqref{BlowUpNorm}).

Our first goal is to obtain  estimates for  $L^p$-norms of $u(t)$  which are uniform in time. Here, we use the Duhamel formula \eqref{duh1} in the following way
\begin{equation}\label{th_lp_1}
\begin{split}
\|u(t)\|_p&\leq \Big{\|}{\rm e}^{\Delta t}u_0+ \int_0^t{\rm e}^{\Delta(t-s)}u(s)(g(u)n-b(n))(s)\dy{s}\Big{\|}_p\\
&+\Big{\|}\int_0^t {\rm e}^{\Delta(t-s)}\nabla\cdot\big(u(s)\nabla\chi(c(s))\big)\dy{s}\Big{\|}_p.
\end{split}
\end{equation}

{\it Step 1: Estimate of $\|u(t)\|_{p_0}$ for each $p_0\in\left(\frac{d}{2},\frac{d}{d-2}\right)$ and $d\in \{2,3\}$.}
 As in Step 1 of the proof of Theorem~\ref{th_inft}, the first term on the right-hand side of inequality \eqref{th_lp_1} with $p=p_0$ will be denoted by $\|v(t)\|_{p_0}$, where $v(x,t)$ is a solution to the auxiliary problem \eqref{rd}--\eqref{rd-ini} with $f=u(g(u)n-b(n))$ and $v_0=u_0$. Recall
\begin{align*}
\|f(t)\|_1\leq\left\|u(t)(g(u)n-b(n))(t)\right\|_1\leq C\|u(t)\|_1 \quad\text{for all}\quad t>0
\end{align*}
because $g$, $b$, and $n$ are bounded.
Hence, using Lemma \ref{lemm_rd} (note that $p_0<\frac{d}{d-2}$), inequality \eqref{mass-ecoli}, and the elementary estimate $\|u_0\|_1\leq C(\Omega)\|u_0\|_{p_0}$, we obtain
\begin{equation}\label{th_lp_2}
\left\|{\rm e}^{\Delta t}u_0+ \int_0^t{\rm e}^{\Delta(t-s)}u(s)(g(u)n-b(n))(s)\dy{s}\right\|_{p_0}\leq C(\|u_0\|_{p_0}+\|n_0\|_1),
\end{equation}
for some constant $C>0$ independent of $t>0$.
{Now, we deal with the second term on the right-hand side of inequality  \eqref{th_lp_1} with $p=p_0$. 
First, 
we use equation \eqref{duh2} and inequalities 
\eqref{G3}, \eqref{G5} to estimate
\begin{equation}\label{est_c}
\begin{split}
\|\nabla c(t)\|_{r}&\leq \left\|\nabla{\rm e}^{(\Delta-\beta)t} c_0\right\|_{r}+ \int_0^t\left\|\nabla {\rm e}^{(\Delta-\beta)(t-s)}u(s)\right\|_{r}\dy{s} \\
&\leq {\rm e}^{-\beta t}\|\nabla c_0\|_{r}+C\int_0^t 
\left(1+(t-s)^{-\frac{d}{2}\left(\frac1{p_0}-\frac{1}{r}\right)-\frac{1}{2}}\right){\rm e}^{-\lambda_1(t-s)}\|u(s)\|_{p_0}\dy{s},
\end{split}
\end{equation}
where  the exponent $r$ satisfies 
\begin{equation}\label{r1}
r\geq p_0\qquad \text{and}\qquad -\frac{d}{2}\left(\frac1{p_0}-\frac{1}{r}\right)-\frac{1}{2}>-1.
\end{equation}
Next, using the heat semigroup estimate from Lemma \ref{lem:LpLq}, the assumption $\chi^\prime\in L^\infty([0,\infty))$ and the H\" older inequality with $\frac{1}{q}=\frac{1}{p_0}+\frac{1}{r}$ we obtain
\begin{equation}\label{est_u}
\begin{split}
\Big\|\int_0^t {\rm e}^{\Delta(t-s)}&\nabla\cdot \big(u(s)\nabla\chi(c(s))\big)\dy{s}\Big\|_{p_0}\\ 
&\leq C\int_0^t\left(1+(t-s)^{-\frac{d}{2}(\frac{1}{q}-\frac{1}{p_0})- \frac{1}{2}}
\right)
{\rm e}^{-\lambda_1(t-s)} \|u\nabla\chi(c(s))\|_q\dy{s}\\
&\leq  C\|\chi^\prime\|_\infty\int_0^t\left(1+(t-s)^{-\frac{d}{2r}-\frac{1}{2}}\right) {\rm e}^{-\lambda_1(t-s)}\|u(s)\|_{p_0}\|\nabla c(s)\|_{2p_0}\dy{s},
\end{split}
\end{equation}
where we require 
\begin{equation}\label{r2}
\frac{1}{p_0}+\frac{1}{r}\le 1 \qquad\text{and} \qquad  -\frac{d}{2r}-\frac{1}{2}>-1.
\end{equation}
By elementary calculations, one can always find $r$ satisfying all conditions in \eqref{r1} and \eqref{r2} under the assumptions $p_0\in\left(\frac{d}{2},\frac{d}{d-2}\right)$ and $d\in \{2,3\}$.}

Now, we define function
\begin{align*}
f(t)\equiv\sup_{0\leq s\leq t}\|u(s)\|_{p_0}
\end{align*}
and 
 by inequality \eqref{est_c}, we have that
\begin{align}\label{est_c2}
\|\nabla c(t)\|_{2p_0}\leq \|\nabla c_0\|_{2p_0}+ Cf(t).
\end{align}

Finally, applying estimates \eqref{th_lp_2}, \eqref{est_u} and \eqref{est_c2} into \eqref{th_lp_1} we obtain
\begin{align*}
\|u(t)\|_{p_0}\leq C\big(\|u_0\|_{p_0}+\|n_0\|_1\big)+ Cf(t)\big(\|\nabla c_0\|_{2p_0}+ Cf(t)\big),
\end{align*}
which leads the following inequality
\begin{align}\label{inq_fbis}
f(t)\leq  C_1(\|u_0\|_{p_0}+\|n_0\|_1)+ C_2\|\nabla c_0\|_{2p_0}f(t)+C_3f^2(t)
\end{align}
for positive constants $C_1$, $C_2$ and $C_3$ independent of $t>0$ and of the solution.
Now, we prove that, for a sufficiently small initial datum, inequality \eqref{inq_fbis} implies that $f(t)$ has to be bounded function.

Indeed, denote $H(y)=C_3y^2+(B-1)y+A$, where $B=C_2\|\nabla c_0\|_{2p_0}$ and $A= C_1(\|u_0\|_{p_0}+\|n_0\|_1)$. It is easy to check that for $4AC_3<(B-1)^2$, the equation $H(y)=0$ has two roots, say $y_1$ and $y_2$. Moreover, for $H^\prime(0)=B-1<0$, those roots are both positive. Hence, since $f(t)$ is non-negative and continuous,
 if we assume that $f(0)=\|u_0\|_{p_0}\in (0,y_1)$ then $f(t)\in [0,y_1]$ for all $t>0$. Note here that $f(0)\leq A$ because we can choose $C_1\geq 1$ without loss of generality. Moreover, by a direct calculation, we have $A<y_1$. Hence, $f(0)\in(0,y_1)$, and this completes the proof of Step 1.

{\it Step 2: Estimate of $\sup_{t>0}\|u(t)\|_\infty$.}{
We come back to inequality \eqref{th_lp_1} with $p=\infty$. Note that $2\in(\frac{d}{2},\frac{d}{d-2})$ for $d\in\{2,3\}$. Hence, by Step 1, we have that $\sup_{t>0}\|u(t)\|_2<\infty$. Thus, we use Lemma \ref{lemm-rdq} with $p=\infty$ and $q=2$ to obtain the following estimate of the first term on the right-hand side of \eqref{th_lp_1}
\begin{align*}
\left\|{\rm e}^{\Delta t}u_0+ \int_0^t{\rm e}^{\Delta(t-s)}u(s)(g(u)n-b(n))(s)\dy{s}\right\|_{\infty}
\leq C\big(\|u_0\|_{p_0}+\|n_0\|_1+\sup_{t>0}\|u(t)\|_2\big).
\end{align*}

Now, we deal with the second term on the right-hand side of \eqref{th_lp_1}. First, we consider the case $d=2$. By Step 1, for each $p\in[1,\infty)$ there is a constant $C>0$ such that $\|u(t)\|_p\leq C$ for all $t>0$. Using relation \eqref{est_c2} we also have that  $\|\nabla c(t)\|_p\leq C$ for all $t>0$ and for each $p\in[1,\infty)$. Hence, by the heat semigroup estimate \eqref{G3} and the H\" older inequality, we obtain the inequalities
\begin{equation*}
\begin{split}
\left\|\int_0^t\nabla {\rm e}^{\Delta(t-s)}u(s)\nabla\chi(c(s))\dy{s}\right\|_\infty&\leq C\int_0^t\left(1+(t-s)^{-\frac{1}{3}- \frac{1}{2}}\right){\rm e}^{-\lambda_1(t-s)} \|u\nabla\chi(c(s))\|_3\dy{s}\\
&\leq  C\|\chi^\prime\|_\infty\int_0^t\left(1+(t-s)^{-\frac{5}{6}}\right) {\rm e}^{-\lambda_1(t-s)}\|u(s)\|_6\|\nabla c(s)\|_6\dy{s},
\end{split}
\end{equation*}
where the right-hand side is bounded uniformly in $t>0$.

Next, we consider the case $d=3$, where by Step 1, we have $\sup_{t>0}\|u(t)\|_p<\infty$ for each $p\in[1,3)$. Hence, using estimate \eqref{G3} and the H\" older inequality, {\it cf.} \eqref{est_c}, we get
\begin{align}\label{est_c3}
\|\nabla c(t)\|_q\leq {\rm e}^{-\beta t}\|\nabla c_0\|_q+C\int_0^t 
\left(1+(t-s)^{-\frac{3}{2}(\frac{1}{p}-\frac{1}{q})-\frac{1}{2}}\right){\rm e}^{-\lambda_1(t-s)}\|u(s)\|_p\dy{s}.
\end{align}
Note, that the function $(t-s)^{-\frac{3}{2}(\frac{1}{p}-\frac{1}{q})-\frac{1}{2}}$ is integrable at $s=t$ for $q<\frac{3p}{3-p}$. Hence, for each $q\in[1,\infty)$ there exists a constant $C>0$ such that $\|\nabla c(t)\|_q\leq C$ for all $t>0$. 

Now, we are in a position to estimate the second term on the right-hand side of inequality  \eqref{th_lp_1} with $p=\infty$, for $d=3$, and we use the same reasoning as in the case $d=2$. First, for every $p\in[1,6)$ we obtain
\begin{equation*}
\begin{split}
\left\|\int_0^t\nabla {\rm e}^{\Delta(t-s)}u(s)\nabla\chi(c(s))\dy{s}\right\|_p\leq C\int_0^t\left(1+(t-s)^{-\frac{3}{2}(\frac{1}{2}-\frac{1}{p})-\frac{1}{2}}\right) {\rm e}^{-\lambda_1(t-s)} \|u\nabla\chi(c(s))\|_2\dy{s}&\\
\leq  C\|\chi^\prime\|_\infty\int_0^t \left(1+(t-s)^{-\frac{3}{2}(\frac{1}{2}-\frac{1}{p})-\frac{1}{2}}\right) {\rm e}^{-\lambda_1(t-s)}\|u(s)\|_{5/2}\|\nabla c(s)\|_{10}\dy{s}&.
\end{split}
\end{equation*}
Since the function $(t-s)^{-\frac{3}{2}(\frac{1}{2}-\frac{1}{p})-\frac{1}{2}}$ is integrable at $s=t$ for each $p<6$, and since $\|u(s)\|_{5/2}$ and $\|\nabla c(s)\|_{10}$ are uniformly bounded in $s>0$, we have proved that for each $p\in[1,6)$ we have $\|u(t)\|_p$ is uniformly bounded for all $t>0$. 

Repeating these estimates for $p=\infty$, we obtain
\begin{equation*}
\begin{split}
\left\|\int_0^t\nabla {\rm e}^{\Delta(t-s)}u(s)\nabla\chi(c(s))\dy{s}\right\|_\infty\leq C\int_0^t\left(1+(t-s)^{-\frac{3}{2}\cdot\frac{1}{4}-\frac{1}{2}}\right) {\rm e}^{-\lambda_1(t-s)} \|u\nabla\chi(c(s))\|_4\dy{s}\\
\leq  C\|\chi^\prime\|_\infty\int_0^t \left(1+(t-s)^{-\frac{7}{8}}\right) {\rm e}^{-\lambda_1(t-s)}\|u(s)\|_{5}\|\nabla c(s)\|_{20}\dy{s},
\end{split}
\end{equation*}
where the right-hand side is uniformly bounded in $t>0$. This completes the proof of Step 2.}

{\it Step 3: Exponential convergence of $(u(t),c(t),n(t),w(t))$.}
First, we show that
\begin{equation}\label{lim:Lp}
 \lim_{t\to\infty}\|u(t)\|_p=0 \qquad  \text{for every} \quad  p\in[1,\infty).
 \end{equation}
  Here, it suffices to combine the standard interpolation inequality of $L^p$-norms
\begin{align}\label{interp}
\|u(t)\|_p\leq C\|u(t)\|_1^{1/p}\|u(t)\|_\infty^{1-1/p},
\end{align}
together with the relation  $\lim_{t\to\infty}\|u(t)\|_1=0$ proved in Theorem \ref{conv_mass} and with the estimate  $\sup_{t>0}\|u(t)\|_\infty<\infty$ by Step 2.

Using relation \eqref{lim:Lp}
 we may show immediately that $\lim_{t\to\infty} \|u(t)\|_\infty=0$ following the reasoning from Step 2 again. Next, we prove the exponential decay of $\|u(t)\|_1$ in the same way as in Step 2 of the proof of Theorem \ref{th_inft}. Therefore, using interpolation equation \eqref{interp} again, we get the exponential decay of $\|u(t)\|_p$ for every $p\in[1,\infty)$ as well. By this fact, one can follow the reasoning from Step 2 once again, to obtain that $\|u(t)\|_\infty\to 0$ exponentially as $t\to\infty$. Moreover, by equation \eqref{duh2} we immediately show the exponential decay of $\|c(t)\|_\infty$.

Finally, to obtain the exponential convergence of $n(t)$ and $w(t)$ 
towards a number $n_\infty$ and a bounded function $w_\infty$, 
it suffices to repeat arguments from Step 6 and 7 of the proof of Theorem~\ref{th_inft}.
\end{proof}

\section{Blow up of solutions }\label{sec:blow}

Now we prove the theorem on a blowup  of solutions to problem  \eqref{eq1-b}--\eqref{eq-ini-b}.

\begin{proof}[Proof of Theorem \ref{blow-up}]
Here, we adapt an analogous  proof of a blow up of solutions to the parabolic-elliptic model of chemotaxis from  the work by  Nagai \cite{N01}. 

For given numbers $r_1$ and $r_2$ satisfying $0 < r_1 < r_2<\text{dist}(q,\partial\Omega)$, we define the function $\phi\in C^1([0,\infty))\cap W^{2,\infty}((0,\infty))$ by the formula
\begin{align*}
\phi(r):=\left\{\begin{aligned}
&r^2 \qquad\ &\textrm{for}&\qquad &0&\leq r\leq r_1,\\
&a_1r^2+a_2r+a_3  &\textrm{for}& &r_1&\leq r\leq r_2,\\
&r_1r_2 \qquad\ &\textrm{for}&\qquad & & r> r_2,
\end{aligned}
\right.
\end{align*}
where
\begin{align*}
a_1=-\frac{r_1}{r_2-r_1},\quad a_2=\frac{2r_1r_2}{r_2-r_1},\quad a_3=-\frac{r_1^2r_2}{r_2-r_1}.
\end{align*}
Thus,  the function $\varphi(x)=\phi(|x|)$ satisfies $\varphi\in C^1(\R^2)\cap W^{2,\infty}(\R^2)$. Moreover, by direct computations, we obtain
\begin{align}
\Delta\varphi(x)=4\quad \text{for}\quad |x|\leq r_1\qquad\text{and}\qquad \Delta\varphi(x)\leq 2\quad\text{for}\quad |x|>r_1.\label{ass3_phi}
\end{align}

Now, we consider a non-negative solution $(u,c,n)$ of problem \eqref{eq1-b}--\eqref{eq-ini-b} on an interval $[0,T_{\text{max}})$ and define mass and the generalized moment for fixed $q\in \Omega$ by the formulas
\begin{align*}
M(t)=\int_\Omega u(x,t)\dx\quad \text{and}\quad I(t)=\int_\Omega u(x,t)\varphi(x-q)\dx.
\end{align*}
Integrating by parts and 
by relation \eqref{ass3_phi} it is clear that
\begin{align}\label{est_dif}
\int_\Omega u(x,t)\Delta\varphi(x-q)\dx\leq 4M(t).
\end{align}
Moreover, since the functions $b$, $g$ and $n$ are bounded and non-negative, we obtain the following estimate
\begin{align}\label{est_re}
\int_{\Omega} (g(u)n-b(u))(x,t)u(x,t)\varphi(x-q)\dx\leq C_3I(t),
\end{align}
where $C_3=G_0\|n_0\|_\infty$.

{
Next, we recall an estimate which is a straightforward adaptation of the result from \cite{N01}.
Let $q\in\Omega$, $0<r_1<r_2<\text{dist}(q,\partial\Omega)$ and $\varphi(x)=\phi(|x-q|)$ be defined as above. Then, for all $t\in(0,T_{\text{max}})$, we have the following estimate
\begin{equation}\label{nagai}
\int_\Omega u(x,t)\nabla\varphi(x-q)\cdot\nabla c(x,t)\dx\leq -\frac{\alpha}{2\pi}M(t)^2+ C_1M(t)I(t)+ C_2M(t)^{3/2}I(t)^{1/2}
\end{equation}
for some constants $C_1$, $C_2$ depending on $r_1$, $r_2$ and $\text{dist}(q,\partial\Omega)$, only.
For the proof of this inequality, it suffices to repeat calculations from \cite[Inequalities (3.2), (3.5), (3.7)-(3.9)]{N01}.

}

Thus, multiplying equation \eqref{eq1-b} by $\varphi(x-q)$, integrating over $\Omega$ and using estimates \eqref{est_dif}--\eqref{est_re} together with inequality  \ref{nagai} we obtain
\begin{align*}
\frac{d}{\dy{t}} I(t)\leq 4M(t)-\frac{\alpha\chi_0}{2\pi}M^2(t)+(C_1\chi_0M(t)+C_3)I(t)+ C_2\chi_0M(t)^{3/2}I(t)^{1/2}.
\end{align*}

Note that for all $s>0$ and $\varepsilon>0$ we have the inequality $s^{1/2}\leq \varepsilon+\frac{1}{4\varepsilon}s$. Hence, for fixed $\varepsilon>0$, which will be chosen later, we use inequality \eqref{mass-ecoli} to obtain
\begin{align}\label{bl-est1}
\frac{d}{\dy{t}} I(t)\leq 4M(t)+\varepsilon-\frac{\alpha\chi_0}{2\pi}M^2(t)+C_4I(t),
\end{align}
where
\begin{align}\label{C}
C_4=C_3+C_1\chi_0(\|u_0\|_1+\frac{1}{\gamma}\|n_0\|_1)+ \frac{C_2^2\chi_0^2(\|u_0\|_1+\frac{1}{\gamma}\|n_0\|_1)^3}{4\varepsilon}.
\end{align}
Estimate \eqref{bl-est1} immediately implies that
\begin{align}\label{est1-I}
\frac{d}{\dy{t}}\Big{(} I(t){\rm e}^{-C_4t}\Big{)}\leq \Big{(}4M(t)+\varepsilon-\frac{\alpha\chi_0}{2\pi}M^2(t)\Big{)}{\rm e}^{-C_4t}.
\end{align}
Next, integrating equation \eqref{eq1-b} over $\Omega$ and using the inequalities $0\leq g(u)n\leq C_3=G_0\|n_0\|_\infty$ and $0<b(t)\leq B_0$, we deduce that
\begin{align*}
\frac{d}{\dy{t}}M(t)\leq C_3 M(t)\qquad\text{and}\qquad  \frac{d}{\dy{t}}M(t)\geq -B_0 M(t),
\end{align*}
hence,
\begin{align}\label{est-M}
M(t)\leq M_0{\rm e}^{C_4t} \qquad\text{and}\qquad
M(t)\geq M_0{\rm e}^{-B_0t} \qquad\text{for all}\qquad t>0.
\end{align}
Substituting estimates \eqref{est-M} in \eqref{est1-I} we obtain the inequality
\begin{align*}
\frac{d}{\dy{t}}\Big{(} I(t){\rm e}^{-C_4t}\Big{)}\leq 4M_0+\varepsilon-\frac{\alpha\chi_0M^2_0}{2\pi}{\rm e}^{-(C_4+2B_0)t},
\end{align*}
which implies
\begin{align}\label{est_mom}
I(t){\rm e}^{-C_4t}\leq I(0)+ (4M_0+\varepsilon)t -\frac{\alpha\chi_0M_0^2}{2\pi(C_4+2B_0)}(1-{\rm e}^{-(C_4+2B_0)t}).
\end{align}

To complete the proof of the nonexistence of global-in-time solutions, it suffices to show that right-hand side of inequality \eqref{est_mom} is negative for some $t>0$. Hence, it suffices to study the function
$
f(t)=A+Bt-D(1-{\rm e}^{-kt}).
$
First, note that $f$ attains its minimum at a certain point if and only if $B<kD$, which is the case if the number $\frac{1}{2\pi}M_0(8\pi-\alpha\chi_0M_0)+\varepsilon$ is negative. Here, one can chose for instance $\varepsilon=\frac{1}{4\pi}M_0(\alpha\chi_0M_0-8\pi)$. Thus, for sufficiently small $f(0)=I(0)=A$ there exist $t>0$ such that $f(t)<0$.

Hence, under these assumptions, the function $I(t)$ becomes negative in a finite time, which is impossible due to positivity of $\int_\Omega u(x,t)\varphi(x)\dx$. This means that a solution $u(t)$ with sufficiently small initial generalized moment $I(0)$ and with the initial mass satisfying $M_0>{8\pi}/(\alpha\chi_0)$ cannot be continued for all $t>0$.
%
\end{proof}


\appendix
\section{Parabolic estimates}

First, we recall estimates on the heat semigroup $\{e^{t\Delta}\}_{t\geq 0}$ in a bounded domain $\Omega$ with the Neumann boundary condition.

\begin{lemm}\label{lem:LpLq}
Let  $\lambda_1>0$ denote the first nonzero eigenvalue of $-\Delta$ in a bounded domain $\Omega\subset \R^d$ under the Neumann boundary conditions. For all  $1 \le q \le p \le +\infty$, there exist constants $C=C_1(p,q,\Omega)$ such that
\begin{itemize}
\item []
\begin{align}\label{G1}
\norm{{e^{t\Delta} f}}{p} \le C\left(1+t^{- \frac{d}{2}\left( \frac{1}{q}-\frac{1}{p} \right)}\right)e^{-\lambda_1t}\norm{f}{q}
\end{align}
for all $f \in L^q (\Omega)$ satisfying $\int_\Omega f\dx=0$ and all $t>0$;
\item []
\begin{align}\label{G2}
\norm{{e^{t\Delta} f}}{p} \le C\big{(}1+ t^{- \frac{d}{2}\left( \frac{1}{q}-\frac{1}{p} \right)}\big{)}\norm{f}{q}
\end{align}
for all $f \in L^q (\Omega)$ and all $t>0$;
\item []
\begin{align}\label{G3}
\norm{{\nabla \left(e^{t\Delta} f\right)}}{p} \le C \left(1+t^{- \frac{d}{2}
 \left(\frac{1}{q}-\frac{1}{p}\right) - \frac{1}{2}}\right)e^{-{\lambda_1} t}\norm{f}{q}
\end{align}
for all $f \in L^q (\Omega)$ and all $t>0$;
\item []
\begin{align}\label{G4}
\norm{{e^{t\Delta} \nabla \cdot f}}{p} \le C \left(1+t^{- \frac{d}{2}
 \left(\frac{1}{q}-\frac{1}{p}\right) - \frac{1}{2}}\right)e^{-{\lambda_1} t}\norm{f}{q},
\end{align}
provided $q>1$, for all $f \in (W^{1,q} (\Omega))^d$ and all $t>0$;
\item []
\begin{align}\label{G5}
\norm{{\nabla e^{t\Delta} f}}{p} \le C e^{-{\lambda_1} t}\norm{\nabla f}{p},
\end{align}
provided $p\in [2,\infty)$, for all $f \in W^{1,p} (\Omega)$ and all $t>0$;
\end{itemize}
\end{lemm}
{
Inequalities \eqref{G1}--\eqref{G5} are well-known in a general case of an analytic semigroup of bounded operators in $L^p(\Omega)$ 
generated by a elliptic partial differential 
operator. 
Some versions of them can be found in the monograph by 
\cite[Lemma 3 on p.~25]{rothe} and in
the abstract theory developed  in \cite{A83}.
Here, we quote refined versions of these estimates proved in \cite{W2, Cao}.


Next, we recall a technical lemma which is used systematically in this work
and we skip its elementary proof, see {\it e.g.}~\cite[Lemma 1.2]{W2}.

\begin{lemm}\label{serre}
Let $k\in [0,1)$ and $M>0$.
For every $f\in L^\infty(0,\infty)$, we have 
$$\sup_{t>0}\int_0^t \left(1+(t-s)^{-k}\right)e^{-M(t-s)}f(s)\dy{s}
\leq C \sup_{t>0} \|f(t)\|_\infty.$$
If
 $\lim_{t\to\infty}f(t)=0$, then
$
\lim_{t\to\infty}\int_0^t \left(1+(t-s)^{-k}\right)e^{-M(t-s)}f(s)\dy{s}=0.
$
Moreover, the speed of decaying of this integral is exponential if the function $f(t)\to 0$ exponentially as $t\to\infty$.
\end{lemm}

}


The following result on the large time behavior of solutions to the nonhomogeneous heat equation seems to be known but we recall its proof for the completeness of the exposition, 
\begin{lemm}\label{lemm_rd}
Let $\Omega\subset\R^d$ be a bounded domain and let
\begin{align*}
 p\in[1,\infty]\quad \text{if}\quad d=1, \qquad
 p\in[1,\infty)\quad \text{if}\quad d=2, \qquad
 p\in\left[1,\frac{d}{d-2}\right)\quad \text{if}\quad d\geq 3.
\end{align*}
Assume that $v_0\in L^p(\Omega)$ and $f=f(x,t)\in L^\infty([0,\infty),L^1(\Omega))$  Then, the solution to the following initial value problem
\begin{alignat}{2}
&v_t=\Delta v+f \quad&&\text{for}\quad x\in\Omega,\ t>0,\label{rd}\\
&\frac{\partial v}{\partial \nu}=0, &&\text{for}\quad x\in\partial\Omega,\ t>0,\\
&v(x,0)=v_0(x) &&\text{for}\quad x\in\Omega\label{rd-ini}
\end{alignat}
satisfies
\begin{align}\label{l-rd-est}
\|v(t)\|_p\leq C\big(\|v_0\|_p+\|v(t)\|_1+\sup_{s>0}\|f(s)\|_1\big) \quad\text{for all}\quad t>0,
\end{align}
where a constant $C$ is independent of $t>0$.
Moreover, if $\|f(\cdot,t)\|_1\to 0$ as $t\to\infty$ then we have
\begin{align}\label{conv_int}
\left\|v(t)-\frac{1}{|\Omega|}\int_\Omega v(t)\dx\right\|_p\to 0\quad \text{as}\quad t\to\infty.
\end{align}
In addition, if  $\|f(\cdot,t)\|_1\to 0$ exponentially as $t\to\infty$, then the convergence in \eqref{conv_int} is exponential, as well.
\end{lemm}
\begin{proof}
The function
\begin{equation}\label{def_w}
w(x,t)=v(x,t)-\frac{1}{|\Omega|}\int_\Omega v(x,t)\dx
\end{equation}
is a solution to the following initial value problem
\begin{align*}
&w_t=\Delta w+f(x,t)-\frac{1}{|\Omega|}\int_\Omega f(x,t)\dx \quad\text{for}\quad x\in\Omega,\ t>0,\\
&w(x,0)=w_0(x)=v_0(x)-\frac{1}{|\Omega|}\int_\Omega v_0(x)\dx,
\end{align*}
supplemented with the Neumann boundary condition. We estimate the $L^p$-norm of $w$ using its Duhamel representation
\begin{align}\label{duh_app}
w(t)=e^{\Delta t}w_0+\int_0^te^{\Delta(t-s)}\left(f-\frac{1}{|\Omega|}\int_\Omega f\dx\right)\dy{s}.
\end{align}
Obviously, we have the inequality 
$
\left\|f(s)-\frac{1}{|\Omega|}\int_\Omega f(x,s)\dx\right\|_1\leq 2\|f(s)\|_1.
$
Thus, we may use estimate \eqref{G1} (note that $\int_\Omega w(x,t)\dx=0$ for all $t\geq 0$) in the following way
\begin{align}\label{lemm_rd1}
\|w(t)\|_p\leq Ce^{-\lambda_1 t}\|w_0\|_p+ C\int_0^t\left(1+(t-s)^{-\frac{d}{2}(1-\frac{1}{p})}\right)e^{-\lambda_1(t-s)}\|f(s)\|_1\dy{s}.
\end{align}
Now, the inequality $-\frac{d}{2}(1-\frac{1}{p})>-1$ holds true due to the assumption on $p$.
Moreover, notice that by the definition of $w$ in \eqref{def_w}, we have the following elementary inequalities
\begin{align}
\|v(t)\|_p&\leq \|w(t)\|_p+|\Omega|^\frac{1-p}{p}\|v(t)\|_1\label{lemm_rd2}\\
\|w_0\|_p&\leq \|v_0\|_p+|\Omega|^\frac{1-p}{p}\|v_0\|_1\leq C\|v_0\|_p.\label{lemm_rd3}
\end{align}
Thus, applying estimates \eqref{lemm_rd2}--\eqref{lemm_rd3} in inequality \eqref{lemm_rd1} we obtain bound \eqref{l-rd-est} because 
$\sup_{t>0}\int_0^t\left(1+(t-s)^{-\frac{d}{2}(1-\frac{1}{p})}\right)e^{-\lambda_1(t-s)}\dy{s}<\infty$.
To show  convergence   \eqref{conv_int}, we apply Lemma \ref{serre} to inequality \eqref{lemm_rd1}.
\end{proof}

%

In this work, we need also another version of estimates from  Lemma \ref{lemm_rd}.

\begin{lemm}\label{lemm-rdq}
Let $\Omega\subset\R^d$ be a bounded domain. Fix $p\in[1,\infty]$. Assume that $v_0\in L^p(\Omega)$ and $f\in L^\infty((0,\infty), L^q(\Omega))$ for some $\frac{dp}{2p+d}\leq q\leq p$. Then there exist a constant $C>0$ such that the solution of problem \eqref{rd}--\eqref{rd-ini} satisfies
\begin{align*}
\|v(t)\|_p\leq C\big(\|v_0\|_p+\|v(t)\|_1+\sup_{t>0}\|f(s)\|_1+\sup_{t>0}\|f(s)\|_q\big)
\end{align*}
for all $t>0$. Moreover, if $\|f(\cdot,t)\|_1\to 0$ and $\|f(\cdot,t)\|_q\to 0$ as $t\to\infty$ then we have
\begin{align}\label{conv_int2}
\left\|v(t)-\frac{1}{|\Omega|}\int_\Omega v(t)\dx\right\|_p\to 0\quad \text{as}\quad t\to\infty.
\end{align}
In addition, if  $\|f(\cdot,t)\|_1\to 0$ and $\|f(\cdot,t)\|_q\to 0$ exponentially as $t\to\infty$, then the convergence in \eqref{conv_int2} is exponential as well.
\end{lemm}

\begin{proof}
We proceed in the same way as in the proof of Lemma \ref{lemm_rd}. The only difference consists in writing inequality \eqref{lemm_rd1} in the following way
\begin{align*}
\|w(t)\|_p\leq Ce^{-\lambda_1 t}\|w_0\|_p+ C\int_0^t\left(1+(t-s)^{-\frac{d}{2}(\frac{1}{q}-\frac{1}{p})}\right)e^{-\lambda_1(t-s)} \Big{(}\|f(s)\|_q+|\Omega|^\frac{1-q}{q}\|f(s)\|_1\Big{)}\dy{s}.
\end{align*}
\end{proof}

\section*{Acknowledgments}
R.~Celi\'nski and G.~Karch were supported by the International Ph.D.
Projects Programme of Foundation for Polish Science operated within the
Innovative Economy Operational Programme 2007-2013 funded by European
Regional Development Fund (Ph.D. Programme: Mathematical Methods in Natural Sciences) and by
the Polish National Science Center grants  No.~2013/09/N/ST1/04316 and No.~2013/09/B/ST1/04412.
M.Mimura was supported by Grant-in-Aid for Exploratory Research No.~15K13462.  D. Hilhorst, M. Mimura and P. Roux acknowledge the support of the CNRS GDRI ReaDiNet.



\end{document}